\documentclass[10pt]{amsart}

\usepackage{amsmath, amsthm, amssymb}
\usepackage{verbatim}
\usepackage{color}
\usepackage[shortlabels]{enumitem}
\usepackage[margin=2.5cm]{geometry}
\usepackage{comment}
\usepackage{soul}
\usepackage{xcolor}

\usepackage{thmtools}
\usepackage{thm-restate}

\allowdisplaybreaks[2]

\usepackage[colorlinks]{hyperref}

\makeatletter
\let\oldhypertarget\hypertarget
\renewcommand{\hypertarget}[2]{%
  \oldhypertarget{#1}{#2}%
    \protected@write\@mainaux{}{%
        \string\expandafter\string\gdef
          \string\csname\string\detokenize{#1}\string\endcsname{#2}%
    }%
  }
\newcommand{\myhyperlink}[1]{%
  \hyperlink{#1}{\csname #1\endcsname}%
  }
\makeatother

\newcommand{\R}{\mathbb{R}}
\newcommand{\N}{\mathbb{N}}

\newcommand{\SL}{\mathcal{L}}

\newcommand{\SO}{\mathcal{O}}
\renewcommand{\S}{\mathbb{S}}
\newcommand{\IP}[2]{\left<#1,#2\right>}

\newtheorem{thm}{Theorem}[section]
\newtheorem{prop}[thm]{Proposition}
\newtheorem{lem}[thm]{Lemma}
\newtheorem{cor}[thm]{Corollary}

\theoremstyle{definition}
\newtheorem{rmk}[thm]{Remark}

\begin{document}

\title{Curve shortening flow with an ambient force field}
\author{Samuel Cuthbertson\and Glen Wheeler\and Valentina-Mira Wheeler}

\thanks{}
\address{University of Wollongong\\
Northfields Ave\\
2522 NSW, Australia}
\email{glenw,vwheeler,csamuel@uow.edu.au}

\begin{abstract}
In this paper we consider the anisotropic curve shortening flow in the plane in
the presence of an ambient force.
We consider force fields in which all their derivatives are bounded in the
$L^{\infty}$ sense.
We prove that closed embedded curves that have a minimum of curvature
sufficiently large shrink to round points.
The method of proof follows along the same lines of Gage and Hamilton, in that
we study a rescaling to prove curvature bounds.
We additionally show that the influence of an ambient force field may make such
a result untrue, by giving sufficient conditions on the ambient field that
ensures eventual non-convexity of an initially convex curve evolving under the
flow.
\end{abstract}
\maketitle

\section{Introduction}

In this paper we consider the anisotropic curve shortening flow $\gamma:\S^1\times[0,T)\rightarrow \R^2$ in the presence of an ambient force field.
The evolution equation is
\begin{equation} \label{flow}
	\begin{cases}
	\partial_t^\bot \gamma
	= (\sigma_1k+\sigma_2 + (V\circ\gamma)\cdot\nu)\nu & \text{ in } \S^1\times(0,T)\text{, and}\\
	 \gamma = \gamma_0& \text{ on }\S^1\times\{t=0\}\,.\\
	 \end{cases}
\end{equation}
Here $\gamma_0$ is a smooth closed embedded curve, $\nu$ is the unit inward
pointing normal, $\sigma_i$ are constants (with $\sigma_1>0$), and
$V:\R^2\rightarrow \R^2$ is a vector field.
At times we term $V$ the \emph{ambient force field} or simply \emph{force field}.
Intuitively we may think of $V$ as incorporating a force external to $\gamma$ that affects its evolution.
Examples include external flows such as atmospheric wind and currents in fluid.
In the present work the ambient force field is fixed in time for simplicity.

As is well-known, the tangential component of the velocity does not change the
geometry of the evolving curve (only its parametrisation).
This is why we take into account only the normal component of the velocity in
\eqref{flow}.

Curvature flows with spatio-temporal forcing terms have received some recent
attention.
A seminal contribution is Takasao-Tonegawa \cite{takasao2016existence}, where
deep regularity results are proved, both up to and after singular times.
Sun \cite{sun2018singularities} studied the blow up limits of forced
mean curvature flow.
Their work (in the setting of Brakke's flow) is a weak formulation of mean
curvature flow (see \cite{kamotion,tonegawa2019brakke}).
Sun used a monotonicity formula to show that blow-up limits satisfy the
self-shrinker equation 
\[
H(x)+\frac{x^\perp}{-2t} = 0\,.
\]
The blow-up limit about a singular point $(y,s)$ is the limit of the
parabolically rescaled solutions given by $M_t^\lambda =
\lambda^{-1}(M_{s+\lambda^2 t}-y)$.

Dirr-Kawali-Yip \cite{dirr2008pulsating} proved the existence of a pulsating
wave when the ambient force field is periodic and defined intrinsically on the
evolving surface. A pulsating surface with velocity $c\neq 0$ is a solution
satisfying
\[
M_t(t+\tau) = M_t(t)+z,\text{ for all } z\in \mathbb{Z}^{n+1}, \text{ }\tau = \frac{\langle \nu, z\rangle}{c}\,.
\]
Their method involved constructing sub and super solutions that satisfy the
flow and are of the form above for the graph equation.

Cesaroni-Novaga-Valdinoci \cite{cesaroni2011curve} included a space-dependent
forcing term to study the curve shortening flow in a heterogeneous medium.
In their work the normal speed is $k+g(x,y)$ where $g$ is a periodic function
defined on a unit square.
The authors analyse the shape of finite time singularities.
Their estimates depend only on the initial curve and the $L^{\infty}$ norm of
$g$.
The graphical case appears to be a good setting to work for this problem as the
equation the graph function satisfies is 
\[
	u_t = \frac{u_{xx}}{1+u_x^2}+g(x,u(x))\sqrt{1+u_x^2}\,.
\]
Their results lead to existence and uniqueness for solutions to the above
equation when the function $g$ does not depend on $u$.

Mikula and \v{S}ev\v{c}ovi\v{c} \cite{mikula2004direct} presented an algorithm
to numerically study the flow.
Furthermore, they applied Angenent's abstract existence result
\cite{angenent1990nonlinear} and ruled out time-periodic solutions through the
use of a Lyapunov functional. 

Of fundamental concern is the relationship between the dynamics of our flow
\eqref{flow} and the anisotropic curve shortening flow without a vector field.
For the vanilla curve shortening flow, Grayson \cite{grayson} proved that all
embedded curves shrink to round points.
For the anisotropic flow, results of \cite{xpz,he2019curvature} show that 
 a class of initial curves continue to be driven to round points.
This includes convex curves, reminiscent of Gage and Hamilton's
\cite{gage1986heat} seminal work on curve shortening flow.
Indeed, many of the techniques in prior work and our paper here are inspired by
Gage and Hamilton.

The vector field $V$ is locally bounded on the plane, which would indicate that
it should not dramatically affect the dynamics of the flow in a neighbourhood
of a small circle where $k$ may be arbitrarily large.
However this does not turn out to be the case.
One simple way to see this is to take $V$ to be a smooth extension of
$-(\sigma_1\hat k + \sigma_2)\hat \nu$ where the hat is used to denote
quantities on some `target' curve.
The flow may evolve so that its curvature field exactly cancels with the
ambient force field, becoming stationary well before shrinking to a point.
This phenomenon leads to an application of the curve shortening flow in the
presence of a force field to Yau's problem, see \cite{SGV3}.

While this is a special case, it is indicative of a general unexpected
phenomenon: Many vector fields $V$ do not allow a Gage-Hamilton result.
To highlight this, in Section \ref{S:loss} we show by example that even strict
convexity can be lost in finite time, for quite a broad class of vector fields
$V$.
The result is as follows.

\begin{restatable}[Preservation of convexity is not generic]{thm}{Tloss}\label{T:loss}
Let $V:\R^2\rightarrow\R^2$, $V\in C^4_{loc}(\R^2)$, be a vector field such
that for some point $p\in\R^2$ and vector $\tau\in\S^1\subset\R^2$ we have
\begin{equation}
\label{EQD2Vhyp}
D^2_{\tau,\tau}V(p)\cdot R\tau < 0\,,
\end{equation}
where $R:\R^2\rightarrow\R^2$ is defined by $R(x,y) = (-y,x)$.

There exists a solution $\gamma:\S^1\times[0,T)\rightarrow\R^2$ to \eqref{flow}
such that (i) $\gamma_0$ is strictly convex; and (ii) for some $t_0\in(0,T)$,
$\gamma(\cdot,t_0)$ is not convex.
\end{restatable}

For a simple example, consider $V(x,y) = (x,-x^2)$.
Then $D^2_{e_1,e_1}V = (0,-2)$ and the LHS of \eqref{EQD2Vhyp} is $-2$.
So \eqref{EQD2Vhyp} is satisfied everywhere in this case.

But the condition \eqref{EQD2Vhyp} only needs to hold at a single point.
Thus this example is easily generalised to fields of the form $V(x,y) =
(f(x),g(x))$ where there exists an $x_0$ such that $f'(x_0) \geq 0$ and
$g''(x_0) <0$.
Due to the requirement \eqref{EQD2Vhyp} being negativity (and not sufficient or
uniform negativity) and only needing to hold at a single point, we call this
loss of convexity \emph{generic}.

To prove Theorem \ref{T:loss} we construct a solution that is initially convex
and becomes non-convex in finite time, using an approach pioneered by Giga-Ito \cite{giga1999loss}.
It leaves open the possibility that solutions with everywhere large, positive curvature, even in the presence of \eqref{EQD2Vhyp}, may nevertheless remain convex.
Our main result confirms that, for three classes of ambient field $V$, this is
essentially correct.
Furthermore, we fully determine the dynamics of these convex solutions.
They all converge to round points.

\begin{thm}[Gage-Hamilton type result]
\label{T:main}
Let $U=B_{R_0}(0)\subset\R^2$ be an open disk of radius $R_0\in(0,\infty]$ centred at the origin.
Let $V:\R^2\rightarrow\R^2$, $V\in C^\infty_{loc}(\R^2)$, $V\in C^\infty(U)$ be a vector field, and set $C_0,C_1,C_2$ to be constants such that
\begin{equation}
\label{EQvhyps}
||V||_{L^\infty(U)} \leq C_0
\,,\quad
||DV||_{L^\infty(U)} \leq C_1
\,,\quad\text{and}\quad
||\IP{D^2_{X,X}V(x,y)}{RX}||_{L^\infty(U\times\S^1)} \leq C_2
\,.
\end{equation}
Assume one of the following \emph{confinement conditions}:
\begin{enumerate}[(a)]
\item (Global boundedness) $R_0=\infty$,
\item (Small $R_0$, dominant diffusion) $R_0 < \frac{\sigma_1}{|\sigma_2| + C_0}$,
\item (Large $R_0$, integrability) The function
\[
	R(r) = ||V||_{L^\infty(B_r(0))}
\]
is uniformly globally Lipschitz, so that the IVP
\begin{equation}
\label{EQivp}
x'(r) = |\sigma_2|+R(x(r))\,,\quad x(0) = r_0\,,
\end{equation}
 where $r_0>0$ has a unique global solution $x:[0,\infty)\rightarrow\R$.
\end{enumerate}

Let $\gamma:\S^1\times[0,T)\rightarrow\R^2$ be a solution to \eqref{flow} such that $\gamma_0(\S)\subset U$.
There exists constants $M\in(0,\infty]$ and $K\in[0,\infty)$ depending only on $\sigma_1$, $\sigma_2$,
$C_0$, $C_1$ and $C_2$ with the following property.
If
\begin{enumerate}[(i)]
\item $\min k(\cdot,0) > K$,
\item $L(\gamma_0) < M$, 
\end{enumerate}
and in case (c) $A(\gamma_0) \le 1$, $R_0 \ge x((\sigma_1\pi)^{-1})$; 
then $\gamma(\mathbb{S},t)\subset U$ for all $t$, and $\gamma(\S,t)$ converges
as $t\rightarrow T$ in the Hausdorff metric on $\R^2$ to a round point $\SO$.
Upon rescaling $\gamma$ to $\hat\gamma:\S\times[0,\infty)\rightarrow\R^2$ via
\[
\hat\gamma(\hat\theta,\hat t)
= \frac1{\sqrt{2T}}e^{\hat t}
    \Big(\gamma\big(\theta, T(1-e^{-2\hat t})\big)-\SO\Big)
\,
\] 
the flow $\hat\gamma$ converges exponentially fast in the $C^\infty$ topology
to a standard round circle.
\end{thm}

We now make a number of natural remarks motivated by our main results.

\begin{rmk}
In the case where $V=0$, we note that Theorem \ref{T:main} generalises the
first part of \cite[Theorem 1.1]{he2019curvature}, which itself contains the
case of vanilla curve shortening flow, where $\sigma_2 = C_0 = C_1 = 0$, and
$\sigma_1=1$.
This is because $K=0$ and $M=\infty$.
In other words the condition (i) becomes just strict convexity and (ii)
becomes bounded length, which is vacuous.
In the case where $V$ is not constant, our theorem is completely new.
\end{rmk}

\begin{rmk}
A condition on the growth of $V$ (such as (c)) is required if we do not confine the initial curve to a small disk centred at the origin (as in (b)).
This highlights an ambient coordinate dependence (a lack of translation invariance specifically) that is not an issue for curve shortening flow (and many other flows).
Even if the length and curvature are arbitrarily small and large respectively,
the curve may blow up in finite time.
For instance, if $V(x,y) = (1+x^2+y^2)^\frac{p}{2}(x,y)$ for $p>0$, then
numerically a circle of any radius (in particular a circle satisfying (i) and
(ii)) can be placed in the plane such that the resultant flow exists at most
only for finite time and in particular leaves every compact set.
For $p\le0$ the integrability condition (c) holds.
\end{rmk}

\begin{rmk}
Confinement condition (b) allows for very general vector fields $V$ with
potentially extreme growth outside of $U$.
However the diffusion is always in balance with the absolute size of $V$ in
$U$, which is the main disadvantage of condition (b).

Confinement condition (c) does not require anything on the diffusion coefficient.
However the growth of $V$ is restricted very far from the origin.
Linear functions yield exponentials for $x$ and it is not possible to go much
further than this.
Nevertheless, near the origin the size of $V$ may be quite large.

One may consider that, a-posteriori, the flow is contained to $U$.
At that point, its growth outside $U$ is completely unimportant.

In neither condition do we need any control over derivatives of $V$.
\end{rmk}

\begin{rmk}
The global boundedness assumption (a) is not subsumed by integrability (c),
since the length condition is slightly different; in particular, in the case of
(c) we need an additional (mild) assumption on initial area.
\end{rmk}

\begin{rmk}
The constants $M$ and $K$ are not abstract and can be written down explicitly.
They depend on the roots of a quadratic and cubic polynomial respectively.
In cases (a) and (b), for $M$, we may take
\[
	M = \max\Bigg\{\frac{\pi\sigma_1\sigma_2+\sigma_1\sqrt{\pi^2\sigma_2^2+3\pi^2C_0^2}}{C_0^2},
	\frac{\pi\sigma_2+\sqrt{\pi^2\sigma_2^2+3\pi^2\sigma_1C_1\left(1-\frac{\sigma_1C_1}{4C_0^2}\right)}}{C_1\left(1-\frac{\sigma_1C_1}{4C_0^2}\right)}\Bigg\}\,
\]
This choice is proved in Lemma \ref{LMlengthbd}.

For case (c), we take
\begin{align*}
M = 
\min\bigg\{ \frac{4\sigma_1\pi}
{|\sigma_2|-\sigma_2 + \sqrt{(|\sigma_2|-\sigma_2)^2 + 4\pi x^2(T_0)}}, 
\frac{\sigma_1\pi}{
		(|\sigma_2|-\sigma_2)/2
			+ x(T_0)
	}
\bigg\}
\,.
\end{align*}
The assumption on $A(\gamma_0)$ is essentially free, since it is quite likely
that $M \le 2\sqrt\pi$ and then $A(0) \le 1$ is guaranteed by the isoperimetric
inequality.
This is the subject of Proposition \ref{PNcasec}.

A discussion on the constant $K$ is given in Lemma \ref{LMkconvex} and Remark \ref{RMKcubic}.
\end{rmk}

\begin{rmk}
It is possible that none of the three confinement conditions are satisfied and yet a curve can still shrink to a point provided we put stronger conditions on the initial curvature.
As an example, take $V(\gamma) = \langle \alpha , \gamma \rangle \gamma$ where $\alpha \in \R^2$ with $|\alpha| = 1$.
Then we have 
\begin{align*}
	|V(X)|^2 &= \langle \langle \alpha, X\rangle X,\langle \alpha, X\rangle X\rangle 
	=|X|^2\langle \alpha,X\rangle^2
\end{align*}
and so $||V||_{L^\infty(B_r(0))} = |\alpha|r^2 = r^2$; the supremum is achieved when $X$ points in the same direction as $\alpha$ and is as large as possible.
Hence the IVP \eqref{EQivp} is $x'(r) = |\sigma_2|+x(r)^2$ with $x(0) = r_0$, and one can easily show that 
\[x(r) \geq \frac{r_0}{1-rr_0}
\] 
so that the solution is not global. However, using the evolution for curvature
we can derive a condition on the initial curve such that the minimum of curvature blows up
before $t = 1/r_0$. See the first author's thesis \cite{sambient2024}
for a detailed proof.
\end{rmk}

\begin{rmk}
There is no requirement that circles remain circular under the flow, and indeed
this is a very strong restriction on the vector field (see \cite{SGV4} for a
study of such vector fields).
For instance, in our example $V(x,y) = (x,-x^2)$, initial circles will not remain
circular, with their furthest horizontal points moving faster southward than
their points closer to the origin.
So, the dynamic picture is that a sufficiently convex curve, even if initially
circular, will become non-circular and only asymptotically approach circularity in
infinite time.
\end{rmk}

\begin{rmk}
Theorem \ref{T:loss} and Theorem \ref{T:main} may apply with the same force
field $V$, i.e. to the same flow.
To illustrate, consider again the simple example of $V(x,y) = (x,-x^2)$.
Theorem \ref{T:loss} applies and says that there exist (many) curves that are
initially strictly convex that become non-convex in finite time.\footnote{We don't know their long-time dynamics beyond this.}
Theorem \ref{T:main} also applies and says that all sufficiently convex small
curves shrink to round points.

In terms of their proofs, the length of the curve used in the proof of Theorem
\ref{T:loss} is not important; however its curvature is.
The curvature must be small relative to the second derivative of $V$, and so
all of the curves that satisfy condition (i) of Theorem \ref{T:main} violate
\eqref{EQepschoiceloss}.
\end{rmk}

\begin{rmk}
Non-preservation of convexity (Theorem \ref{T:loss}) motivates the following
question: what is the dynamics of the flow if we assume a-priori that all
convex curves remain convex?
This question, while related, is quite different to the one answered by our
main result (Theorem \ref{T:main}), and forms the subject of another paper
\cite{SGV2}.
\end{rmk}

\begin{rmk}
It seems intuitively clear that for some choices of $V$, the solution to
\eqref{flow} may be nothing more than a family of rigid motions applied to a
solution of the wound healing flow \cite{he2019curvature}, that is,
\eqref{flow} with $V=0$.
This is true when $V$ is a Killing field, but we are not aware of a general
condition under which this equivalence holds.
We give the details of the Killing field case in Appendix A.
\end{rmk}

The paper is organised as follows.
In Section \ref{S:loss} we give examples of initially strictly convex curves
that lose convexity, completing the proof of Theorem \ref{T:loss}. 
In Section \ref{S:evo} and Section \ref{S:convex} we establish the Hausdorff
convergence part of Theorem \ref{T:main}.
Section \ref{S:speed} is in preparation for the rescaling of Section
\ref{S:rescaling}, in which we use a variety of techniques (including a
monotonicity formula) to conclude smooth convergence of the rescaled flow. We
use the notation $C_i$ to represent the $L^\infty$ norms of the
$i\textsuperscript{th}$ derivative of $V$. We use $D_i$ to denote most other
constants (for aesthetics the count restarts at the beginning of Lemma
$\ref{gradest}$).

\section*{Acknowledgements}

The work of the first author was completed under the financial support of an
Australian Postgraduate Award.
He is grateful for their support.
The work of the third author is partially supported by ARC Discovery Project
DP180100431 and ARC DECRA DE190100379.
She is grateful for their support.

\section{Loss of Strict Convexity is Generic}
\label{S:loss}

In this section we prove Theorem \ref{T:loss}:

\Tloss*

Our proof is inspired by Giga-Ito \cite{giga1999loss}. 

\begin{proof}
Suppose we have a solution $\gamma:\S^1\times[0,T)\rightarrow\R^2$ to the flow \eqref{flow}.
Let $p\in\R^2$ and $\tau\in\S^1$ be a point in the plane and direction respectively.
We consider the flow in a tubular neighbourhood of the line $\SL$ passing through $p$ parallel to $\tau$ in graphical coordinates over the line $\SL$.
Without loss of generality we assume $p = \gamma(0,0)$ and $\tau = \gamma_x(0,0)$ (here we parametrise $\S^1$ by angle).
Choose $\delta>0$ such that $\gamma(\cdot,0)$ is graphical over $\SL$ in an interval $I_\delta = (-\delta,\delta)$.
That is, there exists a function $u_0:I_\delta\rightarrow\R$ such that
\[
    \gamma(x,0) = p + x\tau + u_0(x)R\tau
    \,,\quad x\in I_\delta\,.
\]
Above we used $R\tau$ to denote the counter-clockwise right angle rotation of $\tau$, as in the statement of Theorem \ref{T:loss}.
Consider the interval $I_{\delta/2} = (-\delta/2,\delta/2)$.
There exists a maximal $t_g\in(0,T]$ and parametrisation such that the flow $\gamma$ can be written as a graph as follows:
\[
    \gamma(x,t) = p + x\tau + u(x,t)R\tau
    \,,\quad (x,t)\in I_{\delta/2}\times[0,t_g)\,.
\]
Here $u:I_{\delta/2}\times[0,t_g)\rightarrow\R$ is the graph function.
Note that $u$ is smooth and all derivatives of $u$ are uniformly bounded on $I_{\delta/2}\times[0,t_g]$ (this is why we take $\delta/2$).

We calculate the following equations on $I_{\delta/2}\times[0,t_g)$:
\begin{align*}
	\gamma_x &= \tau+u_xR\tau\\
	\gamma_{xx} & = u_{xx}R\tau\\
	|\gamma_x| &= \sqrt{1+(u_x)^2}\\
	\nu &= \frac{R\tau-u_x\tau}{\sqrt{1+(u_x)^2}}\\
    k &= \Big(\frac{R\tau-u_x\tau}{\sqrt{1+(u_x)^2}}\Big)
        \cdot
        \Big(
           \frac{\gamma_{xx}}{1+(u_x)^2}
           - \gamma_x\frac{u_xu_{xx}}{\sqrt{1+u_x^2}^3} 
        \Big)
    = \frac{u_{xx}}{\sqrt{1+(u_x)^2}^3}
    \,.
\end{align*}
Note that the unit tangent vector to the graph is $\frac{\tau+u_xR\tau}{\sqrt{1+(u_x)^2}}$, not $\tau$.
The unit normal vector to the graph is $\nu$.
Then
\begin{align*}
\frac{u_t}{\sqrt{1+(u_x)^2}}\nu
    &= (u_t R\tau\cdot\nu)\nu
    = \gamma_t^\bot
    = (\sigma_1k+\sigma_2 + V(\gamma)\cdot\nu)\nu
\\
    &= 
        \bigg(\frac{\sigma_1u_{xx}}{\sqrt{1+(u_x)^2}^3}
        + \sigma_2 + V(\gamma)\cdot \nu\bigg)\nu
        \,,
\end{align*}
which implies
\[
u_t = \frac{\sigma_1u_{xx}}{1+(u_x)^2}+\sqrt{1+(u_x)^2 }(\sigma_2 + V(\gamma)\cdot\nu)
\,
.
\]
Let $h(x,t) = \sqrt{1+(u_x)^2}(V(\gamma)\cdot\nu)(x,t)$.
We see that 
\[
u_{xt} = \frac{\sigma_1u_{x^3}}{1+(u_x)^2}-\frac{2\sigma_1(u_{xx})^2u_x}{(1+(u_x)^2)^2}+\frac{\sigma_2u_{x}u_{xx}}{\sqrt{1+u_x^2}} + h_x
\,,
\]
and
\begin{align*}
u_{xxt}
=
  \frac{\sigma_1u_{x^4}}{1+(u_x)^2} 
&- \frac{6\sigma_1u_{x^3}u_xu_{xx}}{(1+(u_x)^2)^2}
- \frac{2\sigma_1(u_{xx})^3}{(1+(u_x)^2)^2}
\\&\qquad + \frac{8\sigma_1(u_x)^2(u_{xx})^3}{(1+(u_x)^2)^3}
+ \frac{\sigma_2(u_{xx})^2}{\sqrt{1+u_x^2}}
+ \frac{\sigma_2u_xu_{x^3}}{\sqrt{1+u_x^2}}
- \frac{\sigma_2u_x^2u_{xx}^2}{\sqrt{(1+u_x^2)^3}}
+ h_{xx}
\,.
\end{align*}
Notice that $h = V(\gamma)\cdot(R\tau - u_x\tau)$.
We calculate
\begin{align*}
h_x
&= D_{\gamma_x}V(\gamma) \cdot (R\tau-u_x\tau)
    - u_{xx}V(\gamma)\cdot \tau
\\
h_{xx}
&= (D^2_{\gamma_x,\gamma_x}V(\gamma) + D_{\gamma_{xx}}V(\gamma))\cdot
        (R\tau - u_x\tau)
    - (2u_{xx}D_{\gamma_x}V(\gamma) + u_{x^3}V(\gamma))\cdot\tau
    \,.
\end{align*}
Now let us choose $u_0(x) = \varepsilon x^2$.
For $\varepsilon>0$ we may close $\gamma_0$ so that $k_0>0$ everywhere (in fact we may insist that $\inf k_0 > \frac{2\varepsilon}{\sqrt{1+\varepsilon^2\delta^2}^3}$).

We have 
\begin{align*}
u_{x^l}(0,0) &= 0\text{ for $l\ne2$,}\ u_{xx}(0,0) = 2\varepsilon
\,,
\\
\gamma(0,0) &= p\,,\quad \gamma_x(0,0) = \tau\,,\quad \gamma_{xx}(0,0) = 2\varepsilon R\tau\,,
\\
h(0,0) &= V(p)\cdot R\tau\,,\quad h_x(0,0) = D_\tau V(p)\,,\quad \text{and}
\\
h_{xx}(0,0) &= (D^2_{\tau,\tau}V(p) + 2\varepsilon D_{R\tau}V(p))\cdot R\tau - 4\varepsilon D_\tau V(p)\cdot\tau
\,.
\end{align*}
Thus
\begin{align}
u_{xxt}(0,0)
&= -2\sigma_1(u_{xx}(0,0))^3
    + \sigma_2(u_{xx}(0,0))^2
    + h_{xx}(0,0)
\label{eq:kt}\\
&= -16\sigma_1\varepsilon^3
    + 4\sigma_2\varepsilon^2
    + D^2_{\tau,\tau}V(p)\cdot R\tau
    - 4\varepsilon D_{\tau}V\cdot \tau
    + 2\varepsilon D_{R\tau}V\cdot R\tau
\nonumber
\\
&= \varepsilon(-16\sigma_1\varepsilon^2
    + 4\sigma_2\varepsilon
    - 4 D_{\tau}V\cdot \tau
    + 2 D_{R\tau}V\cdot R\tau)
    + D^2_{\tau,\tau}V(p)\cdot R\tau
\,.
\nonumber
\end{align}
Hypothesis \eqref{EQD2Vhyp} implies that there exists a $C_V>0$ such that
\[
    D^2_{\tau,\tau}V(p)\cdot R\tau = -C_V < 0\,.
\]
Combining this with \eqref{eq:kt} we find
\[
u_{xxt}(0,0)
= \varepsilon(-16\sigma_1\varepsilon^2
    + 4\sigma_2\varepsilon
    - 4 D_{\tau}V\cdot \tau
    + 2 D_{R\tau}V\cdot R\tau)
    - C_V
    \,.
\]
Therefore
\[
u_{xxt}(0,0)
< -\frac{C_V}{2}
\]
provided we choose 
\begin{equation}
\label{EQepschoiceloss}
\varepsilon < \varepsilon_1 := \min\left\{
    1,
    \frac{C_V}{2}(4\sigma_2 + 4 |D_{\tau}V(p)\cdot \tau| + 2 |D_{R\tau}V(p)\cdot R\tau|)^{-1}
    \right\}\,.
\end{equation}
There exists a $K_\varepsilon = K(\varepsilon, D^{(m)}V(p))$ for $m\in\{1,2,3,4\}$ such that
$|u_{xxtt}(0,0)| \le K_\varepsilon$.
By continuity, there exists a $t_{K_\varepsilon}$ such that
$|u_{xxtt}(0,t)| \le 2K_\varepsilon$ for $t<t_{K_\varepsilon}$.
Furthermore, take $K = \sup\{K_\varepsilon\,:\,\varepsilon\in[0,\varepsilon_1]\}$ and $t_K =
\inf\{t_{K_\varepsilon}\,:\,\varepsilon\in[0,\varepsilon_1]\}$.
Note that we also implicitly have a $\delta_0>0$ and $t_{G}>0$ such that the
solution, for any $\varepsilon\in[0,\varepsilon_1]$, is locally smooth and
uniformly graphical on $I_{\delta_0/2}\times[0,t_G]$.
However for $\varepsilon = 0$ the initial data is not strictly convex, so we
must not take $\varepsilon=0$ as our initial data that proves the theorem.

By the fundamental theorem of calculus we have the following crucial estimate
\begin{equation}
\label{EQnonconv}
u_{xx}(0,t) = u_{xx}(0,0) + u_{xxt}(0,0)t
+ \int_0^t \int_0^{\hat t} u_{xxtt}(0,\tilde t)\,\, d\tilde{t}\, d\hat{t} 
< 2\varepsilon - \frac{C_V}{2}t + Kt^2
\,,
\end{equation}
 for $t \in [0,\min(t_K,t_g))$.
The roots of the RHS are
\[
 \frac{C_V/2 \pm\sqrt{C_V^2/4-8\varepsilon K}}{2K}\,.
\]
We see that as $\varepsilon\searrow0$, the first root $t_1$ approaches zero (in time).
Therefore there exists an $\varepsilon_0>0$ such that
\[
0 < t_1 < \min(t_K,t_G)\,.
\]
By \eqref{EQnonconv}, $u_{xx}(0,t_1) < 0$, and so the solution with initial
data corresponding to $\varepsilon_0>0$ is initially strictly convex, and at
time $t_1$ is not convex.
This finishes the proof.
\end{proof}

\begin{rmk}
For curve shortening flow, the above approach does not yield any result.
In that case, $u_{xxt}(0,0)$ can not be made negative independent of
$\varepsilon$, and the right hand side of \eqref{EQnonconv} does not have a
positive root.
\end{rmk}


\section{Evolution Equations, Controlling Length and Curvature}
\label{S:evo}

In this section we calculate the evolution of important geometric quantities
under the flow.
We also prove that a bound on the curvature implies a bound on all derivatives
of curvature.

\begin{prop}\label{p:evo}
For a smooth family of curves $\gamma:\mathbb{S}^1\times[0,T)\rightarrow
\mathbb{R}^2$ evolving under \eqref{flow} we have the following evolution equations.
\begin{itemize}
\item The arc-length element:
\[
	|\gamma_u|_t = -(\sigma_1k^2+\sigma_2)|\gamma_u|+\langle D_\tau V,\tau \rangle |\gamma_u|\,.
\]
\item The commutator:
\[
	\partial_t\partial_s = \partial_s\partial_t+(\sigma_1k^2+\sigma_2k)\partial_s-\langle D_\tau V, \tau \rangle \partial_s\,.
\]
\item The tangent vector:
\[
	\tau_t = (\sigma_1k_s +\langle D_\tau V,\nu \rangle)\nu\,.
\]
\item The normal vector:
\[ 
	\nu_t = -(\sigma_1k_s +\langle D_\tau V,\nu \rangle) \tau\,.
\]
\item The arc-length measure:
\[ 
	\partial_t(ds) = -\sigma_1k^2 -\sigma_2kds+\langle D_\tau V, \tau \rangle ds\,.
\]
\item The length:
\[
	L' = -\sigma_1\int_{\gamma}k^2ds-2\pi \sigma_2-\int_{\gamma}k \langle V, \nu \rangle ds\,.
\]
\item The signed enclosed area:
\[
	A' = -2\pi \sigma_{1} -\sigma_{2}L-\int_{\gamma} \langle V(\gamma), \nu \rangle ds\,.
\]
\end{itemize}
\end{prop}
\begin{proof}
First we calculate
\begin{align*}
	2|\gamma_u||\gamma_u|_t = \partial_t|\gamma_u|^2 &= 2 \langle \gamma_{ut}, \gamma_{u} \rangle \\
	& = 2\langle ((\sigma_1k+\sigma_2)\nu)_u+(V\circ\gamma(u))_u , \gamma_{u} \rangle\\
	& = 2|\gamma_u|^2\langle (\sigma_1k+\sigma_2)_s\nu +(\sigma_1k+\sigma_2)\nu_s+(V\circ\gamma(u))_s, \gamma_{s} \rangle\\
	& = -2|\gamma_u|^2k(\sigma_1k+\sigma_2)+2|\gamma_u|^2 \langle D_\tau V, \tau \rangle\,.\\
	\implies |\gamma_u|_t & = -(\sigma_1k^2+\sigma_2k)|\gamma_u| +|\gamma_u|\langle D_\tau V, \tau \rangle\,.
\end{align*}
Then use this to calculate the commutator
\begin{align*}
	\partial_t\partial_s = \partial_t(|\gamma_u|^{-1}\partial_u) &= -|\gamma_u|^{-2}|\gamma_u|_t\partial_u+|\gamma_u|^{-1}\partial_u\partial_t\\
	& = -\frac{1}{|\gamma_u|^2}(-(\sigma_1 k^2+\sigma_2k)|\gamma_u|+|\gamma_u|\langle D_\tau V, \tau \rangle )\partial_u +\partial_s\partial_t\\
	&= \partial_s\partial_t+(\sigma_1k^2+\sigma_2k)\partial_s - \langle D_\tau V, \tau \rangle \partial_s\,.
\end{align*}
For the tangent vector we calculate
\begin{align*}
	\tau_t =(\gamma_s)_t & = (\gamma_t)_s+(\sigma_1k^2+\sigma_2)\gamma_s-\langle D_\tau V, \tau \rangle \gamma_s \\
	&= ((\sigma_1k+\sigma_2)\nu +V\circ\gamma)_s+(\sigma_1k^2+\sigma_2k)\tau -\langle D_\tau V, \tau \rangle\tau \\
	&= \sigma_1k_s\nu+(\sigma_1k+\sigma_2)\nu_s+D_\tau V +(\sigma_1k^2+\sigma_2k
  )\tau -\langle D_\tau V, \tau \rangle \tau \\
	&= \sigma_1k_s \nu +D_\tau V-\langle D_\tau V, \tau \rangle \tau \\
	&= (\sigma_1k_s+\langle D_\tau V, \nu \rangle)\nu\,.
\end{align*} 
The final line is justified by 
\[
D_\tau V = \langle D_\tau V, \tau \rangle  \tau +\langle D_\tau V, \nu \rangle \nu 
\]
since $\{ \nu,\tau \}$ forms an orthonormal basis\,.\\
We also have 
\[ \nu_t = \langle \nu_t,\nu \rangle \nu + \langle \nu_t, \tau \rangle \tau\,.
\]
Differentiating the facts that $\langle \nu,\nu \rangle = 1$ and $\langle \nu,\tau \rangle = 0$ imply respectively that 
\[ \langle \nu_t,\nu \rangle = 0 \text{ and } \langle \nu_t,\tau \rangle = -\langle \nu,\tau_t \rangle = -\sigma_1k_s-\langle D_\tau V,\nu \rangle\,.
\]
From this we conclude that
\[ \nu_t = -(\sigma_1k_s+ \langle D_\tau V,\nu \rangle)\tau\,.
\]
For the arc-length element we calculate
\begin{align*} \partial_t(ds)  = \partial_t(|\gamma_u|du) &= |\gamma_u|_tdu\\
	&= -(\sigma_1k^2+\sigma_2k)|\gamma_u|du+ \langle D_\tau V, \tau \rangle |\gamma_u|du\\
	&= -(\sigma_1k^2+\sigma_2k)ds+ \langle D_\tau V, \tau \rangle ds\,.\\
\end{align*}
We use this to calculate the time derivative of the length functional
\begin{align*} L'(\gamma(\cdot,t)) = \int_{\gamma} |\gamma_u|_tdu &= -\int_{\gamma} (\sigma_1k^2+\sigma_2k)ds +\int_{\gamma} \langle D_\tau V,\tau \rangle ds \\
	& = - \sigma_1\int_{\gamma}k^2ds -\sigma_2 \int_{\gamma} kds +\int_{\gamma} \langle D_\tau V,\tau \rangle ds\\
	&=  - \sigma_1\int_{\gamma}k^2ds- 2 \pi \sigma_2 + \int_{\gamma} \langle D_\tau V,\tau \rangle ds\,. \\
\end{align*}
Recall that the area inside a closed curve is given by 
\[ A(\gamma(\cdot,t)) = -\frac{1}{2} \int_{\gamma} \langle \gamma , \nu \rangle ds,
\]
therefore 
\begin{align*}
	A' &= -\frac{1}{2} \int_{\gamma} (\langle \gamma_t, \nu \rangle + \langle \gamma, \nu_t \rangle)ds+ \langle \gamma, \nu \rangle \partial_t(ds)\\
	&= -\frac{1}{2} \int_{\gamma} (\langle (\sigma_1k+\sigma_2)\nu+V, \nu \rangle +\langle \gamma, -(\sigma_1k_s+\langle D_\tau V,\nu\rangle )\tau \rangle +\langle \gamma, \nu \rangle (-\sigma_1k^2-\sigma_2k+\langle D_\tau V,\tau \rangle ))ds \\
	&= -\frac{1}{2} \int_{\gamma} (\sigma_1k+\sigma_2 + \langle V, \nu \rangle + \langle \gamma, -\sigma_1k_s\tau \rangle -\langle D_\tau V,\nu\rangle\langle \gamma,\tau \rangle +\langle \gamma, -(\sigma_1k+\sigma_2)k\nu \rangle+ \langle \gamma,\nu \rangle\langle D_\tau V,\tau \rangle)ds\\
	&= -\frac{1}{2} \int_{\gamma} (\sigma_1k+\sigma_2 + \langle V, \nu \rangle + \langle \gamma, -\sigma_1k_s\tau-(\sigma_1k+\sigma_2)\tau_s \rangle -\langle D_\tau V,\nu\rangle\langle \gamma,\tau \rangle + \langle \gamma,\nu \rangle\langle D_\tau V,\tau \rangle)ds\\
	& = -\frac{1}{2} \int_{\gamma} (\sigma_1k+\sigma_2 + \langle V, \nu \rangle + \langle \gamma, -((\sigma_1k+\sigma_2)\tau)_s \rangle -\langle D_\tau V,\nu\rangle\langle \gamma,\tau \rangle + \langle \gamma,\nu \rangle\langle D_\tau V,\tau \rangle)ds\\
	&=-\frac{1}{2} \int_{\gamma} (\sigma_1k+\sigma_2 + \langle V, \nu \rangle +\langle \gamma_s, (\sigma_1k+\sigma_2)\tau \rangle -\langle D_\tau V,\nu\rangle\langle \gamma,\tau \rangle + \langle \gamma,\nu \rangle\langle D_\tau V,\tau \rangle)ds\\
	&= -\int_{\gamma} (\sigma_1k+\sigma_2)ds -\frac{1}{2} \int_{\gamma} \langle V, \nu \rangle ds +\frac{1}{2}\int_{\gamma}\langle D_\tau V,\nu\rangle\langle \gamma,\tau \rangle ds -\frac{1}{2} \int_{\gamma} \langle \gamma,\nu \rangle\langle D_\tau V,\tau \rangle ds\,.
\end{align*}
Using integration by parts (noting that $\frac{\partial}{\partial s} V(\gamma(s)) = D_\tau V$) we calculate the third integral to be 
\[
\frac{1}{2}\int_{\gamma}\langle D_\tau V,\nu\rangle\langle \gamma,\tau \rangle ds= -\frac{1}{2} \int_{\gamma} \langle V, -k\tau \rangle \langle \gamma,\tau \rangle+ \langle V, \nu \rangle \langle \tau, \tau \rangle +\langle V, \nu \rangle \langle \gamma, k\nu \rangle ds\,.
\]
The fourth can be written as
\[ -\frac{1}{2}\int_{\gamma} \langle \gamma , \nu \rangle \langle D_\tau V , \tau \rangle ds = \frac{1}{2}\int_{\gamma} (\langle \gamma , \nu \rangle \langle V , k\nu \rangle + \langle \tau , \nu \rangle \langle V , \tau \rangle + \langle \gamma , -k\tau \rangle\langle V , \tau \rangle)ds\,.
\]
From this we observe the cancellation of several terms and thus conclude the claim:
\[
A' = -2\pi\sigma_{1} -\sigma_{2}L-\int_{\gamma} \langle V(\gamma), \nu \rangle ds\,. 
\]
\end{proof}

For bounding length and preserving convexity we will need uniform bounds on
$V$, as well as its derivatives.
These bounds can be assumed (that is confinement condition (a)), but to allow
more general (and natural) force fields $V$ that are only locally bounded, we
present the two confinement conditions (b) and (c).
Each of these enable us to confine the flow to a compact subset of
the plane, which gives uniform estimates on $V$ and its derivatives.
Before we give the proofs, let us remark that while the avoidance principle
does hold for our flows, it does not help, since we do not have access to a
suitable comparison solution.  That is, in one sense, a main purpose of this
present paper.

First, let us consider confinement condition (b), when we assume that the
diffusion coefficient is relatively large.
In this case the flow is not only bounded to a compact region, its position
vector is bounded by the maximum of the position vector of the initial data.

\begin{lem}
\label{LMconfcondb}
Let $\gamma:\S^1\times[0,T)\rightarrow\R^2$ be a solution to \eqref{flow}
satisfying the hypotheses of Theorem \ref{T:main} and confinement condition (b).
Then $\gamma(\S^1,t)\subset\subset U$ for all $t$.
\end{lem}
\begin{proof}
We calculate
\[
	(|\gamma|^2)_t
	= 
		2(\sigma_1k+\sigma_2 + (V\circ\gamma)\cdot\nu)\nu\cdot\gamma
\]
so
\[
	|\gamma|_t
	= 
		(\sigma_1k+\sigma_2 + (V\circ\gamma)\cdot\nu)\nu\cdot\hat\gamma
	\,.
\]
where $\hat\gamma = \gamma/|\gamma|$.
At a non-zero max for $|\gamma|$, we have $\gamma\cdot\tau = 0$.
(If the max of $|\gamma|$ is zero, then the curve is a point, and a point is not a solution to our flow.)
Furthermore
\[
|\gamma|_{ss} = (\gamma\cdot\tau/|\gamma|)_s
              = (\gamma\cdot\tau)_s/|\gamma|
              	+ \gamma\cdot\tau(|\gamma|^{-1})_s
              = (1 + k\gamma\cdot\nu)|\gamma|^{-1}
              	+ \gamma\cdot\tau(|\gamma|^{-1})_s
\,.
\]
At a maximum we have
\[
(1 + k\gamma\cdot\nu)|\gamma|^{-1} \le 0
\]
or
\[
k\gamma\cdot\nu \le -1\,.
\]
From the evolution equation, at a new max, we have
\begin{align*}
0 &\le (\sigma_1k+\sigma_2 + (V\circ\gamma)\cdot\nu)\nu\cdot\hat\gamma
\\
\Longrightarrow\qquad 
0 &\le (\sigma_1k+\sigma_2)\nu\cdot\gamma + ((V\circ\gamma)\cdot\nu)\nu\cdot\gamma
\\
\Longrightarrow\qquad 
0 &\le -\sigma_1 + (\sigma_2 + ((V\circ\gamma)\cdot\nu))\nu\cdot\gamma
\end{align*}
Since $\gamma_0(\S)\subset U\subset\subset\R^2$, a new max of $|\gamma|$ must
occur in $\overline U$ and so in particular at the new max $|\gamma|\le R_0$ and $|V|\le C_0$.
Thus
\[
\sigma_1 \le R_0(|\sigma_2| + C_0)\,.
\]
This contradicts the hypothesis.
\end{proof}

Now we consider confinement condition (c).
This condition enables us to control the growth of the maximum of $V$ along the
flow in terms of the global solution $x$ to the ODE $x'(r) = |\sigma_2| + R(x(r))$, where
$R(r)$ is the supremum of $|V|$ on a disk of radius $r$ and $x(0)$ equals the
maximum of $|\gamma_0|$.

\begin{prop}
\label{PNcasec}
Let $\gamma:\S^1\times[0,T)\rightarrow\R^2$ be a solution to \eqref{flow}
satisfying the hypotheses of Theorem \ref{T:main} and confinement condition (c).
Then length is monotone decreasing for all $t$ and the maximal smooth existence time $T$ satisfies $T\le (\sigma_1\pi)^{-1}$.
Furthermore
\[
\gamma(\S,t)\subset B_{x((\sigma_1\pi)^{-1})}(0)
\]
for all $t$, where $x:[0,\infty)\rightarrow\R$ is the growth function given in
confinement condition (c).
\end{prop}
\begin{proof}
The proof of Lemma \ref{LMconfcondb} above shows that the function $m(t) = \max
|\gamma(\cdot,t)|$ satisfies
\[
m'(t) \le |\sigma_2| + R(m(t))
\,.
\]
By ODE comparison we thus have $m(t) \le x(t)$ and so
\begin{equation}
\label{EQposvecest}
|\gamma|(s,t) \le m(t) \le x(t)\,.
\end{equation}
We calculate
\begin{align*}
A'(t)
	&= -2\sigma_1\pi - L(t)\sigma_2 - \int_\gamma \IP{V(\gamma)}{\nu}\,ds
\\
	&\le -2\sigma_1\pi + L(t)\frac12(|\sigma_2|-\sigma_2) + L(t)x(t)
\,,
\end{align*}
so, for $T_0\le T$,
\begin{equation}
\label{EQareaest}
A(T_0) \le A_0 - 2\sigma_1\pi T_0 
	+ \frac12(|\sigma_2|-\sigma_2)\int_0^{T_0} L(t)\,dt
	+ \int_0^{T_0} L(t)x(t)\,dt
\,.
\end{equation}
For length, we have
\[
	L' = -\sigma_1\int_{\gamma} k^2ds-\sigma_2\int_{\gamma} kds - \int_{\gamma} k\langle V, \nu \rangle ds
\,.
\]
Thus
\begin{align*}
    L' &\leq -\sigma_1\int_{\gamma}k^2ds
	 - 2\pi\sigma_2-\int_{\gamma}k\langle V,\nu \rangle ds
\\
    & \leq -\frac{\sigma_1}{2}\int_{\gamma}k^2ds
	 + \pi(|\sigma_2|-\sigma_2)
	 + \frac{L(t)x^2(t)}{2\sigma_1}
\,.
\end{align*}
Now we use the inequality
\begin{equation*}\label{EQintk}
    \int k^2ds \geq \frac{4\pi^2}{L}
\end{equation*}
 which follows from the Poincar\'e inequality applied to the tangent vector.
Rearranging we find 
\begin{equation*}
    (\log L)' \le 
     -2\sigma_1\pi\bigg( \pi L^{-2}
		 - \frac{|\sigma_2|-\sigma_2}{2\sigma_1} L^{-1}
		 - \frac{x^2(t)}{4\sigma_1^2}
	\bigg)
\,.
\end{equation*}
The quadratic on the RHS is non-positive exactly when
\begin{align}
L^{-1}(t) &\ge \frac{\frac{|\sigma_2|-\sigma_2}{2\sigma_1} + \sqrt{\frac{(|\sigma_2|-\sigma_2)^2}{4\sigma_1^2} + 4\pi\frac{x^2(t)}{4\sigma_1^2} }}{
		2\pi
		}
\notag
\\
 &= \frac{|\sigma_2|-\sigma_2 + \sqrt{(|\sigma_2|-\sigma_2)^2 + 4\pi x^2(t)}}{
		4\sigma_1\pi
		}
\,.
\label{EQlengthcondc}
\end{align}
Given condition \eqref{EQlengthcondc} for $t=0$, the length decreases monotonically.
However the right hand side of \eqref{EQlengthcondc} increases and so the length is only
guaranteed to be decreasing for a finite amount of time.
In particular, length will be decreasing until $t=t_0$, characterised by $L^{-1}(t_0)$
being equal to the RHS in \eqref{EQlengthcondc} with $t$ set to $t_0$.

Suppose $T_0\in(0,T]$ is given.
(We will set it momentarily.)
Let us assume that $L(0) = L_0$ is small enough so that \eqref{EQlengthcondc}
is satisfied precisely until $t = T_0$.
Thus $L(t) \le L_0$ for $t\in[0,T_0]$.
Using \eqref{EQlengthcondc}, this is achieved by assuming that
\begin{equation}
L(0) \le
 \frac{4\sigma_1\pi}
{|\sigma_2|-\sigma_2 + \sqrt{(|\sigma_2|-\sigma_2)^2 + 4\pi x^2(t)}}
\,.
\label{EQlengthcondcT0}
\end{equation}
Under the assumption \eqref{EQlengthcondcT0} we estimate the RHS of \eqref{EQareaest} to find
\begin{align*}
A(T_0) &\le A_0 - 2\sigma_1\pi T_0 
	+ \frac12(|\sigma_2|-\sigma_2)T_0L_0
	+ T_0L_0x(T_0)
\\&\le A_0 
	+ T_0\bigg(- 2\sigma_1\pi 
		+ L_0\bigg(\frac12(|\sigma_2|-\sigma_2)
			+ x(T_0)
		\bigg)
	\bigg)
\,.
\end{align*}
There are several choices to be made.
Let us enforce $T_0 = 1/(\sigma_1\pi)$, and then take
\begin{align}
L(0) \le
\min\bigg\{ \frac{4\sigma_1\pi}
{|\sigma_2|-\sigma_2 + \sqrt{(|\sigma_2|-\sigma_2)^2 + 4\pi x^2(T_0)}}, 
\label{EQlengthcondcF}
\frac{\sigma_1\pi}{
		(|\sigma_2|-\sigma_2)/2
			+ x(T_0)
	}
\bigg\}
\,.
\end{align}
Under \eqref{EQlengthcondcF} we find
\[
A(T_0) \le A_0 - \sigma_1\pi T_0 = A_0 - 1\,.
\]
Assuming finally that $A_0 \le 1$, we find $A(T_0) \le 0$ and since the flow is a family of embeddings the maximal time of existence $T$ must satisfy $T\le T_0$.

Thus $\gamma(\S,t)\subset B_{x(T_0)}(0) = B_{x((\sigma_1\pi)^{-1})}(0)$ as required.
\end{proof}

Now we give the length bound for confinement conditions (a) and (b).

\begin{lem}
\label{LMlengthbd}
Let $\gamma:\S^1\times[0,T)\rightarrow\R^2$ be a solution to \eqref{flow}
satisfying the hypotheses of Theorem \ref{T:main} and confinement condition (a) or (b).
Then $L'(t) \le 0$ and in particular $L(t) \le L(\gamma_0)$.
\end{lem}
\begin{proof}
Letting $\alpha \in [0,1)$, we write the length evolution as follows
\[
	L' = -\sigma_1\int_{\gamma} k^2ds-\sigma_2\int_{\gamma} kds -\alpha\int_{\gamma} k\langle V, \nu \rangle ds-(1-\alpha)\int_{\gamma} k\langle V, \nu \rangle ds\,.
\]
For the third integral we have 
\[ 
	\bigg|\alpha\int_{\gamma} k\langle V, \nu \rangle\bigg| = \bigg|-\alpha\int_{\gamma} \langle D_{\tau}V, \tau \rangle\bigg| \leq \alpha LC_1\,.
\]
Thus
\begin{align*}
    L' &\leq -\sigma_1\int_{\gamma}k^2ds-2\pi\sigma_2+\alpha LC_1 -(1-\alpha)\int_{\gamma}k\langle V,\nu \rangle ds\\
    & \leq (\varepsilon(1-\alpha)-\sigma_1)\int_{\gamma}k^2ds-2\pi\sigma_2+\alpha LC_1 +\frac{(1-\alpha)}{4\varepsilon}C_0^2L\\
    & \leq -\frac{3\sigma_1}{4}\int_{\gamma}k^2ds-2\pi\sigma_2+\alpha LC_1+\sigma_1^{-1}(1-\alpha)^2C^2_0L
\end{align*}
with the choice of $\varepsilon = \frac{\sigma_1}{4(1-\alpha)}$, for $\alpha\ne1$.

Again we use the inequality \eqref{EQintk} and rearrange to find 
\begin{equation*}
    \frac{1}{2}(L^2)' \leq L^2\left(\alpha C_1+\sigma_1^{-1}(1-\alpha)^2C_0^2\right)-2\pi\sigma_2L-3\pi^2\sigma_1\,.
\end{equation*}
The quadratic on the RHS is negative exactly when
\begin{equation}
\label{EQlengthalpha}
L(\gamma_0) < \frac{\pi\sigma_2+\sqrt{\pi^2\sigma_2^2+3\pi^2\sigma_1(\alpha C_1+\sigma_1^{-1}(1-\alpha)^2C_0^2)}}{\alpha C_1+\sigma_1^{-1}(1-\alpha)^2C_0^2}
\,.
\end{equation}
The maximum of the RHS occurs for $\sigma_1C_1 < 2C_0^2$ at $\alpha = 1-\frac{\sigma_1C_1}{2C_0^2}$ and
for $\sigma_1C_1 \ge 2C_0^2$ at $\alpha = 0$, which gives precisely condition (ii)
of Theorem \ref{T:main}.
This condition, if satisfied at initial time, ensures length starts decreasing,
which then preserves \eqref{EQlengthalpha} for all future times.
This finishes the proof.
\end{proof}

\begin{rmk}
The length estimate is sharp in the following sense.
If $V$ is a constant vector field, then the maximum in Theorem \ref{T:main}
(ii) is $\infty$ and the condition is vacuous.
Thus for $C_1=0$ there is no length restriction in Theorem \ref{T:main},
regardless of the value of $\sigma_2$.
If $V$ is not a constant vector field there is no reason to expect a length
bound to hold.
For example, if $V(x,y) = (x,y)$ and the curve $\gamma_0$ is a
circle centred at the origin with radius $r_0$ then it remains circular and
\[
	L'(t) = -\frac{2\pi\sigma_1}{r(t)} + 2\pi (r(t)-\sigma_2)
	\,.
\]
Clearly for $r_0$ large enough length (and radius) increases without bound, and
so there is no way to bound length in general.
(For further study of flows that preserve circularity, we refer to
\cite{SGV4}.)
\end{rmk}

The evolution of the curvature is of major importance in studying the flow.

\begin{lem}
\label{LMcurvevo}
Let $\gamma:\S^1\times[0,T)\rightarrow\R^2$ be a solution to \eqref{flow}
satisfying the hypotheses of Theorem \ref{T:main}.
Then
\[
k_t
 = \sigma_1 k_{ss} + \sigma_1 k^3 + \sigma_2k^2 - k(2\langle D_{\tau}V,\tau \rangle-\langle D_\nu V, \nu \rangle)+\langle D^2_{\tau, \tau}V, \nu \rangle
\,.
\]
\end{lem}
\begin{proof}
Using Proposition \ref{p:evo} we calculate
\begin{align*}
k_t 
	&= \langle \gamma_{ss}, \nu \rangle_t =\langle \gamma_{sst}, \nu \rangle + \langle \gamma_{ss}, \nu_t \rangle = \langle \gamma_{sst}, \nu \rangle
\\
	&= \langle (\gamma_s)_{ts} +(\sigma_1k^2+\sigma_2k)\gamma_{ss}-\langle D_\tau V, \tau \rangle \gamma_{ss}, \nu \rangle,
\\
	&= \langle(\tau_t)_s, \nu \rangle +\sigma_1k^3+\sigma_2k^2-k\langle D_{\tau}V, \tau \rangle 
\\
	&= \langle ((\sigma_1k_s+\langle D_{\tau}V, \nu \rangle)\nu)_s, \nu \rangle+\sigma_1k^3+\sigma_2k^2-k\langle D_\tau V, \tau \rangle 
\\
	&= \langle (\sigma_1k_{ss}+\langle D^2_{\tau \tau} V, \nu \rangle -k\langle D_\tau V, \tau \rangle+k\langle D_\nu V, \nu \rangle  )\nu,\nu\rangle+\sigma_1k^3+\sigma_2k^2-k\langle D_\tau V, \tau \rangle
\\
	&= \sigma_1 k_{ss} +\sigma_1 k^3 +\sigma_2k^2-k(2\langle D_{\tau}V,\tau \rangle-\langle D_{\nu}V, \nu \rangle)+\langle D^2_{\tau, \tau}V, \nu \rangle
\,.
\end{align*}
\end{proof}

Lemma \ref{LMcurvevo} and a standard maximum principle argument allows us to argue that if the curvature is large enough initially then it remains so.

\begin{lem}
\label{LMkconvex}
Let $\gamma:\S^1\times[0,T)\rightarrow\R^2$ be a solution to \eqref{flow} 
There exists a constant $K>0$ depending only on $\sigma_1$, $\sigma_2$, $C_1$ and $C_2$ such that if 
\[
\min k(\cdot,0)  > K
\]
 then $k(\cdot,t)\geq K$ for all $t\in [0,T)$.
\end{lem}
\begin{proof}
At a minimum $k_0$ for curvature Lemma \ref{LMcurvevo} implies  
\begin{equation}
\label{EQpdecons}
 \partial_t k-\sigma_1\partial_{ss}k
 \geq \sigma_1k_0^3-\frac12(|\sigma_2|-\sigma_2)k_0^2-3C_1k_0-C_2
\,.
\end{equation}
The right hand side is a cubic in $k_0$.
Set
\[
P(x) =  \sigma_1x^3-\frac12(|\sigma_2|-\sigma_2)x^2-3C_1x-C_2
\,.
\]
Since $\sigma_1>0$ the cubic $P(x)$ has one or two non-negative roots.
Take $K$ to be the largest non-negative root of $P(x)$; then at any new minimum
value we have $k_0>K$ and so the RHS of \eqref{EQpdecons} is non-negative,
contradicting the existence of the new minimum as is standard.
\end{proof}

\begin{rmk}
\label{RMKcubic}
We may find the root $K$ using the standard process for solving a cubic.
The cubic $P(x)$ can be written in depressed cubic form
\[
t^3+pt+q = 0
\]
with $x = t+\frac{|\sigma_2|}{3\sigma_1}$,
\begin{align*}
    p& =\frac{-(36\sigma_1C_1+(|\sigma_2|-\sigma_2)^2)}{12\sigma_1^2}\,,
	\\
    \text{ and } q&= \frac1{27\sigma_1^3}\bigg(
			-\frac14(|\sigma_2|-\sigma_2)^3-\frac{27}{2}\sigma_1(|\sigma_2|-\sigma_2)C_1-27\sigma_1^2C_2
		\bigg)
\,.
\end{align*}
The discriminant is given by 
$\Delta = -(4p^3+27q^2)$, and the root $K$ by either
\[
K = \frac{(|\sigma_2|-\sigma_2)}{6\sigma_1}+2\sqrt{\frac{-p}{3}}\cos\left[\frac{1}{3}\arccos\left(\frac{3q}{2p}\sqrt{\frac{-3}{p}}\right)\right]
\]
or
\[
K = \frac{(|\sigma_2|-\sigma_2)}{6\sigma_1}-2\frac{|q|}{q}\sqrt{\frac{-p}{3}}\text{cosh}\left[\frac{1}{3}\text{arccosh}\left(\frac{-3|q|}{2p}\sqrt{\frac{-3}{p}}\right)\right]
\,.
\]
These formulae give a more explicit view of the dependence of $K$ on $\sigma_1, \sigma_2, C_1$ and $C_2$.
\end{rmk}

As expected, the turning number is invariant along the flow.
This follows from the fact that the flow is a smooth one-parameter family, but
as the explicit calculation is not lengthy, we present it below.

\begin{lem}\label{l:winding}
Let $\gamma:\S^1\times[0,T)\rightarrow\R^2$ be a solution to \eqref{flow}.
Then
\[
	\frac{d}{dt} \int_{\gamma}k\,ds = 0\,.
\]
\end{lem}
\begin{proof}
We calculate
\begin{align*}
\frac{d}{dt} \int_{\gamma} kds &= \int_{\gamma} k_tds+\int_{\gamma} k\frac{d}{dt}ds
	\\
	&= \int_\gamma (\sigma_{1} k_{ss} +\sigma_1 k^3 +\sigma_2k^2-k(2\langle D_{\tau}V,\tau \rangle-\langle D_\nu V, \nu \rangle)+\langle D^2_{\tau, \tau}V, \nu \rangle)ds
	\\&\qquad - \int_{\gamma}(\sigma_1 k^3 +\sigma_2k^2-k\langle D_{\tau}V,\tau \rangle
	)ds\\
	&= \int_{\gamma} \frac{\partial}{\partial s}\langle D_\tau V(\gamma(s)),\nu\rangle ds\\
	&=0\,.
\end{align*}
\end{proof}

An explicit bound on the maximal existence time is given for confinement condition (c) in Proposition \ref{PNcasec}.
For confinement conditions (a) and (b) we are able to obtain $T<\infty$ by using the length decay as follows.

\begin{lem}
\label{LMfintietime}
Let $\gamma:\S^1\times[0,T)\rightarrow\R^2$ be a solution to \eqref{flow} satisfying Theorem \ref{T:main}.
Then the maximal time of existence is finite.
\end{lem}
\begin{proof}
For confinement condition (c) this already follows from Proposition \ref{PNcasec}.
Otherwise, recalling the proof of Lemma \ref{LMlengthbd} we have
\begin{align*}
L'(t) 
    & \leq -\delta\int_{\gamma}k^2ds
\end{align*}
since the inequality $L(\gamma_0) < M$ is strict (here $\delta = M-L(\gamma_0)$).
Integrating gives 
\[
\delta\int_0^T\int_\gamma k^2ds\,dt \leq L(0)-L(T) \le L(0)\,.
\]
If the time of existence is infinite, then the above inequality and Lemma \ref{LMkconvex} yield a contradiction.
Thus $T$ must be finite.
\end{proof}

We note that a precise estimate for $T$ under confinement conditions (a) and
(b) depend on how close the hypothesis $L(\gamma_0) < M$ is to an equality, and
on how large $\min k(\cdot,0)$ is.
For our purposes here, we do not require a precise estimate.

We shall need the following evolution equation for $k_s$.

\begin{lem}
\label{LMksevo}
Let $\gamma:\S^1\times[0,T)\rightarrow\R^2$ be a solution to \eqref{flow}.
Then
\[
    k_{st} = \sigma_1k_{sss}+4\sigma_1k^2k_s+3\sigma_2kk_s+k_sA_1
    +k^2A_2+kA_3+\langle D^3_{\tau, \tau, \tau} V, \nu \rangle\,,
\]
where
\begin{align*}
A_1 &= \langle D_\nu V, \nu \rangle-3\langle D_\tau V, \tau \rangle\\
A_2 &= 3 \langle D_\tau V, \nu \rangle-3\langle D_{\nu}V, \tau \rangle\\
A_3&=3\langle D^2_{\tau, \nu}V , \nu \rangle-3\langle D^2_{\tau, \tau}V, \tau \rangle
\,.
\end{align*}
\end{lem}
\begin{proof}
Using the commutator we calculate
\begin{align*}
(k_s)_t &= (k_t)_s+(\sigma_1k^2+\sigma_2k)k_s-\langle D_\tau V, \tau \rangle k_s\\
	&= \sigma_1k_{sss}+3\sigma_1k^2k_s+2\sigma_2kk_s-k_s(2\langle D_\tau V, \tau \rangle-\langle D_\nu V, \nu \rangle)\\
	&\qquad-k(\langle 2D^2_{\tau, \tau} V, \tau \rangle+2k\langle D_\nu V, \tau \rangle+2k\langle D_\tau V, \nu\rangle -\langle D^2_{\tau,\nu} V, \nu \rangle+k\langle D_\tau V, \nu \rangle+k\langle D_\nu V, \tau \rangle  )\\
	&\qquad+\langle D^3_{\tau,\tau,\tau} V, \nu \rangle+2k\langle D^2_{\tau,\nu} V, \nu \rangle-k\langle D^2_{\tau,\tau} V, \tau \rangle\\
	&\qquad+\sigma_1k^2k_s+\sigma_2kk_s-\langle D_\tau V, \tau \rangle k_s\\
	&=\sigma_1k_{sss}+4\sigma_1k^2k_s+3\sigma_2kk_s-k_s(3\langle D_\tau V, \tau \rangle-\langle D_\nu V, \nu \rangle)\\
	&\qquad-k^2(3\langle D_\nu V, \tau \rangle+3\langle D_\tau V, \nu \rangle)\\
	&\qquad-k(3\langle D^2_{\tau, \tau} V, \tau \rangle-3\langle D^2_{\tau, \nu} V, \nu \rangle)\\
	&\qquad+\langle D^3_{\tau,\tau,\tau} V, \nu \rangle
	\,.
\end{align*}
\end{proof}

Lemma \ref{LMksevo} is the base case for the following more general evolution equation.

\begin{prop}
Let $\gamma:\S^1\times[0,T)\rightarrow\R^2$ be a solution to \eqref{flow}.
The evolution of the $p\textsuperscript{th}$ derivative of $k$ is given by
\[
 (k_{s^p})_t = \sigma_1k_{s^{p+2}} + (p+3)\sigma_1k^2k_{s^p} +(p+2)\sigma_2 kk_{s^p} +A_pk_{s^p}+B_p,
\]
where 
\[
	A_p = \langle D_\nu V, \nu \rangle -(2+p)\langle D_\tau V, \tau \rangle
\]
and $B_p$ depends on $V$ and its derivatives up to order $(p+2)$, and $k_{s^{m}}$ for $0\le m\le p-1$.
\end{prop}
\begin{proof}
The proof is by induction.
The base case $p = 1$ is proven above.

Using the commutator $[\partial_s,\partial_t]$ and the inductive hypothesis we have 
\begin{align*}
(k_{s^{p+1}})_t &= \partial_s \partial_t(k_{s^p})+ k(\sigma_1k+\sigma_2)\partial_s(k_{s^p})-\langle D_\tau V, \tau \rangle\partial_sk_{s^p}
\\
&= \sigma_1k_{s^{p+3}}+2(p+3)\sigma_1kk_sk_{s^p}+(p+3)\sigma_1k^2k_{s^{p+1}}+(p+2)\sigma_2k_sk_{s^p}+(p+2)\sigma_2kk_{s^{p+1}}
\\
&\qquad+(A_p)_sk_{s^p}+(A_p)k_{s^{p+1}}+(B_p)_s+\sigma_1k^2k_{s^{p+1}} +\sigma_2kk_{s^{p+1}}-\langle D_\tau V, \tau \rangle k_{s^{p+1}}
\,.
\end{align*}
Since $B_p$ contains derivatives of $k$ up to the $p-1$\textsuperscript{th} order then $\partial_sB$ contains them up to the  $p$\textsuperscript{th} order.
Thus 
\[
	(k_{s^{p+1}})_t = \sigma_1k_{s^{p+3}}+(\sigma_1(p+3)+\sigma_1)k^2k_{s^{p+1}}+(\sigma_2(p+2)+\sigma_2)kk_{s^{p+1}}+A_{p+1}k_{s^{p+1}}+B_{p+1}.
\]
\end{proof}

For the curve shortening flow, a `bootstrapping principle' applies: a bound on
all derivatives of the curvature up to order $(p-1)$ implies a bound on the
next derivative of curvature.
We wish to prove the same holds here.
For this it is crucial that we have confined the flow to a compact subset of
the plane.

\begin{prop}
\label{PNbootstrap}
Let $\gamma:\S^1\times[0,T)\rightarrow\R^2$ be a solution to \eqref{flow} satisfying Theorem \ref{T:main}.
Let $c_m$, $m\in\N_0$, be such that
\[
||D^mV||_{L^\infty(U)} = c_m\,.
\]
Suppose $D_i$, $i = 0,1,...,p-1$ are such that 
\[
	|k_{s^{i}}| \leq D_{i} \text{ for all } i = 0,1,...,p-1;
\]
then $|k_{s^{p}}| \leq D_p $. 
\end{prop}
\begin{proof}
First we note that due to (any of) the containment conditions the flow is
trapped inside $U$ for all time, and so we may use $c_m$ to estimate $D^mV$
without issue.

We use induction.
For the base case assume that the curvature is bounded i.e $k \leq
D_0 = k_{\text{max}}$.
Let $X = e^{\alpha t}k_s$ where $\alpha$ is a negative
constant to be chosen.
The evolution is therefore
\begin{align*}
(k_s)_t  &= e^{-\alpha t}(X_t-\alpha X)\\
\implies e^{-\alpha t}(X_t-\alpha X) &= e^{-\alpha t}(\sigma_1 X_{ss}+ 4\sigma_1k^2X+3\sigma_2kX+A_1X) + A_2k^2+A_3k+\langle D^3_{\tau, \tau, \tau} V, \nu \rangle\\
\implies X_t-\sigma_1X_{ss} &= X(\alpha +4\sigma_1k^2+3\sigma_2k+A_1) +e^{\alpha t}(A_2k^2+A_3k+\langle D^3_{\tau, \tau, \tau} V, \nu \rangle)\\
\implies (\partial_t-\sigma_1\partial_{ss})X^2 &= 2X(X_t-\sigma_1X_{ss})-2\sigma_1(X_s)^2\\
&\le 2X^2(\alpha+4\sigma_1k^2+3\sigma_2k+A_1)+ 2Xe^{\alpha t}(A_2k^2+A_3k+\langle D^3_{\tau, \tau, \tau} V, \nu \rangle).
\end{align*}
We can assume that $|X| \leq |X|^2$, as otherwise the proof of the base case is finished.
Then
\[
	(\partial_t-\sigma_1\partial_{ss})X^2 \leq 2X^2(\alpha+4\sigma_1k^2_{\text{max}}+3\sigma_2k_{\text{max}}+4c_1+6c_1 k^2_{\text{max}}+6c_2 k_{\text{max}}+c_3).
\]
Take $\alpha < -(4\sigma_1k^2_{\text{max}}+3\sigma_2k_{\text{max}}+4c_1+6c_1 k^2_{\text{max}}+6c_2 k_{\text{max}}+c_3)$ then by the maximum principle 
\[
 |k_s| \leq e^{-\alpha T}\max\limits_{t = 0} |k_s|.
\]
Finite existence time (Lemma \ref{LMfintietime}) finishes the proof of the base case.

The proof for the inductive case follows along the same lines. Assume $|k_{s^{i}}| \le D_i$ for each $i = 1,2,..p-1$.
Let $X = e^{\alpha t}k_{s^p} $ (where $\alpha$ is negative) then 
\[
 (k_{s^p})_t = e^{-\alpha t}(X_t-\alpha X)\,,
\]
and 
\begin{align*} 
e^{-\alpha t }(X_t-\alpha X)
 &= e^{-\alpha t}(\sigma_1X_{ss}+(p+3)\sigma_1k^2X+(p+2)\sigma_2kX+A_pX)+B_p\\
\implies X_t-\sigma_1X_{ss} &= X(\alpha +(p+3)\sigma_1k^2+(p+2)\sigma_2k+A_p) + e^{\alpha t }B_p\\
\implies (\partial_t-\sigma_1\partial_{ss})X^2  &= 2X(X_t-\sigma_1X_{ss})-2\sigma_1(X_s)^2\\
&\le 2X^2(\alpha +(p+3)\sigma_1k^2+(p+2)\sigma_2k+A_p)+2Xe^{\alpha t}B_p-2\sigma_1(X_s)^2
\,.
\end{align*}
Again, assume that $|X| \leq |X|^2$ because the proof is done if otherwise; then
\[
 (\partial_t-\sigma_1\partial_{ss})X^2 \leq 2X^2(\alpha + (p+3)\sigma_1k^2 +(p+2)\sigma_2k+(3+p)c_1+||B_p||_{L^\infty}).
\]
Take $\alpha < -((p+3)\sigma_1k^2 +(p+2)\sigma_2k+(3+p)c_1+||B_p||_{L^\infty})$ and then the maximum principle implies that there exists a $D_{p+1}$ such that
\[
 |k_{s^{p+1}}| \leq D_{p+1}
\]
where $D_{p+1} = e^{-\alpha T}D_p$.
Note that here again we use finite maximal time (Lemma \ref{LMfintietime}).
\end{proof}

\section{Convex Curves Shrink to a Point}
\label{S:convex}

We continue on to prove that initial curves satisfying the assumptions of Theorem \ref{T:main} shrink to points.
For this we follow Gage-Hamilton \cite{gage1986heat}.
Briefly, we change to the angle
parametrisation and use a new time variable $t'$ so that $(t',\theta)$ are
independent.
Then, we establish analogues of Gage-Hamilton's \cite{gage1986heat} ``geometric'',
``integral'' and ``pointwise'' estimates.

First we calculate the curvature evolution in this parametrisation.
\begin{lem}
\label{LMkevothetaparam}
Let $\gamma:\S^1\times[0,T)\rightarrow\R^2$ be a solution to \eqref{flow} satisfying Theorem \ref{T:main}.
Then
\[
 k_{t'} = \sigma_1k^2k_{\theta \theta}+\sigma_1k^3+\sigma_2k^2-k(2\langle D_TV,T \rangle-\langle D_NV,N\rangle)-k_\theta\langle D_TV,N \rangle+\langle D^2_{T,T}V, N\rangle\,,
\]
where $t'$ and $\theta$ are new variables that are independent of each other
and $\theta$ is the angle the tangent makes with a fixed vector in
$\mathbb{R}^2$.
The vectors $T$ and $N$ are the tangent and normal respectively in the $(t',\theta)$ parametrisation.
\end{lem}
\begin{proof}
Recall that the curvature evolves via
\[
 k_t = \sigma_1k_{ss}+\sigma_1k^3+\sigma_2k^2-k(2\langle D_\tau V, \tau \rangle- \langle D_\nu V, \nu \rangle)+\langle D^2_{\tau, \tau}V, \nu \rangle.
\] 	
Under the new parametrisation we have $\frac{\partial \theta}{\partial s} = k$ so that
\[
k_s = kk_\theta, \quad\text{ and }\quad
 k_{ss} = k\left(\frac{\partial}{\partial_\theta}kk_\theta\right) = kk^2_\theta+k^2k_{\theta \theta}\,.
\]
Using the chain rule we have 
\[
 k_t = \frac{\partial k}{\partial t'} \frac{\partial t'}{\partial t}+\frac{\partial k}{\partial \theta} \frac{\partial \theta}{\partial t}. 
\] 
Furthermore $\tau_t = \tau_\theta \frac{\partial \theta}{\partial t} = \nu
\frac{\partial \theta}{\partial t}$ and from Proposition \ref{p:evo} we compute 
\[
 \tau_t = (\sigma_1k_s+\langle D_\tau V, \nu \rangle)\nu.
\] 
Hence $\frac{\partial \theta}{\partial t} = \sigma_1kk_\theta+\langle D_TV,N\rangle$ and we take $\frac{\partial t'}{\partial t}= 1$. Therefore 
\[
 \sigma_1k^2k_{\theta \theta}+\sigma_1kk^2_{\theta}+\sigma_1k^3+\sigma_2k^2-k(2\langle D_TV,T\rangle-\langle D_NV,N\rangle)+\langle D^2_{T, T}V, N \rangle = k_{t'}+\sigma_1kk_{\theta}^2+k_{\theta}\langle D_TV,N\rangle.
\]
Rearranging gives the result.
\end{proof}

For the rest of the paper we abuse notation by reverting back to $t$ to represent $t'$.

\subsection{Geometric estimate}

This estimate, due to Gage-Hamilton, controls the median curvature in terms to the ratio of length to area.
The median curvature $k^*$ is defined by 
\[
 k^* = \sup \{ b : k(\theta)> b \text{ on some interval of length } \pi\}.
\]
\vspace{-5mm}
\begin{lem}(Geometric estimate.)
\label{LMgeomest}
The relation $k^*(t)<\frac{L}{A}$ holds for convex curves.
\end{lem}

For the convenience of the reader, we give a sketch of the proof.
If $k^*(t)>M$ (for some $M>0$), the curve $\gamma$ restricted to an open
interval $(a,a+\pi)$ has curvature $k>M$.
This means that one can contain this
arc in a circle of radius $\frac{1}{M}$ and thus between two parallel lines
$\frac{2}{M}$ units apart.
Convexity now implies that the entire curve lies between these two lines.
In particular, the curve can be contained in a rectangular box with width
$\frac{2}{M}$ and length $\frac{L}{2}$, thus the area inside the curve
satisfies $A(\gamma) < \frac{L}{M}$, which implies $M < L/A$.
Since $M$ can be taken arbitrarily close to $k^*$, this is enough to finish the proof.

\subsection{Integral estimate}

The integral estimate controls the \emph{entropy} of the evolving curve
assuming control on the median of curvature.

\begin{lem} (Integral estimate.) 
\label{LMintegralest}
Let $\gamma:\S^1\times[0,T)\rightarrow\R^2$ be a solution to \eqref{flow}
satisfying the hypotheses of Theorem \ref{T:main}.
Suppose $k^*(t) < K^* < \infty$ for all $t$.
Then
\begin{align*}
\int_0^{2\pi}\log k d\theta
&\le 
\int_0^{2\pi}\log k d\theta\bigg|_{t=0}
	+ \left(2K^* + (C_1G_0)\frac{|\sigma_2|+\sigma_2}{2\sigma_1}\right)
	  	\left(L(0)-L'(t)\right)
\\ &\qquad
	+ \left(2K^* + (C_1G_0)\frac{|\sigma_2|+\sigma_2}{2\sigma_1}\right)
		\left(2\pi(|\sigma_2|-\sigma_2) + 2\pi C_0 + G_1\right)T
\,.
\end{align*}
where $G_0 = \max\left\{\frac{1}{e},\left| \frac{\log K}{K} \right| \right\}$
and $G_1 = 2\pi\sigma_1(K^*)^2+6\pi C_1+2\pi C_2\left(\frac{1}{K} +
G_0\right)$.
\end{lem}
\begin{proof}
We calculate 
\begin{align*}
\frac{d}{dt} \int_0^{2\pi}\log kd\theta
 &= \int_0^{2\pi} \left(\sigma_1kk_{\theta \theta}+\sigma_1k^2+\sigma_2k-(2\langle D_TV,T\rangle - \langle D_NV,N \rangle ) - \frac{k_{\theta}}{k} \langle D_TV,N \rangle +\frac{1}{k} \langle D^2_{T,T}V,N\rangle \right)d\theta
\\
 &\le \int_0^{2\pi}\sigma_1(k^2-k_{\theta}^2)d\theta
	+ \frac12(|\sigma_2|+\sigma_2)\int_0^{2\pi}k d\theta
	+ 6\pi C_1
	- \int_0^{2\pi}(\log k)_\theta \langle D_TV,N\rangle d\theta
	+ \frac{2\pi}{K}C_2
\,.
\end{align*}
We deal with the first integral by splitting the domain into regions where the
curvature is greater than $k^*$.
That is, let 
\[
 W = \{k(\theta) \geq k^* \}.
\]
From the definition of $k^*$, $W$ can be written as the union of intervals $I_i$ of length no greater
than $\pi$.
On each $I_i$ we estimate (using Gage-Hamilton's Wirtinger inequality)
\begin{align*}
\int_{I_i}(k-k^*)^2d\theta
 &\leq \int_{I_i}k_{\theta}^2d\theta
\\
\implies\quad 
 \int_{I_i}(k^2-k^2_\theta )d\theta &\le 2k^*\int_{I_i}kd\theta -2\int_{I_i}(k^*)^2d\theta
\\
&\leq 2K^*\int_{I_i}kd\theta.
\end{align*}
Taking their union we have 
\[
 \sigma_1 \int_W(k^2-k^2_{\theta})d\theta \leq 2\sigma_1K^*\int_0^{2\pi}kd\theta.
\]
On the complement set $W^c$ we have $k \leq k^*$ therefore 
\[
 \sigma_1 \int_{W^c}(k^2-k_{\theta}^2)d\theta \leq \sigma_1\int_0^{2\pi}k^2d\theta \leq 2\pi \sigma_1 (K^*)^2.
\]
Thus we obtain the estimate 
\[
 \sigma_1 \int_{0}^{2\pi}(k^2-k_{\theta}^2)d\theta \leq 2\sigma_1K^*\int_0^{2\pi}kd\theta+2\pi\sigma_1(K^*)^2.
\]
Therefore 
\[
 \frac{d}{dt}\int_0^{2\pi}\log k d\theta
 \leq \left(2K^*+\frac{|\sigma_2|+\sigma_2}{2\sigma_1}\right)\sigma_1\int_0^{2\pi}kd\theta
	+ 2\pi\sigma_1(K^*)^2
	+ 6\pi C_1
	+ \frac{2\pi}{K}C_2 
	- \int_0^{2\pi} (\log k)_\theta \langle D_TV,N\rangle d\theta
\,.
\]
Integration by parts applied to the last integral yields
\[
 -\int_0^{2\pi}(\log k)_\theta \langle D_TV,N\rangle d\theta
= \int_0^{2\pi}\log k\frac{\partial}{\partial \theta} \langle D_TV,N\rangle d\theta.
\]
Using $T_\theta = N$, $N_\theta = -T$, and $\gamma_\theta = \frac{T}{k}$ we have 
\[
 \frac{\partial}{\partial \theta} \langle D_TV,N\rangle = \langle D^2_{T,\gamma_\theta}V,N\rangle +\langle D_NV,N\rangle -\langle D_TV,T\rangle.
\]
This gives 
\begin{align*}
 \int_0^{2\pi} \log k \frac{\partial}{\partial \theta} \langle D_TV,N\rangle d\theta
 &= \int_0^{2\pi} \frac{\log k}{k}\langle D^2_{T,T}V,N\rangle d\theta+\int_0^{2\pi}\log k(\langle D_NV,N\rangle -\langle D_TV,T\rangle)d\theta \\ 
	& \leq C_2\int_0^{2\pi}\left| \frac{\log k}{k} \right| d\theta +2C_1 \int_0^{2\pi}|\log k|d\theta.
\end{align*}
Since $0 < K < k$ 
\[
 \left| \frac{\log k}{k} \right| \leq \left| \frac{\log k_0}{k_0} \right| \text{ for } K \leq 1,
\]
otherwise we have $\frac{\log k}{k} \leq \frac{1}{e}$.
Hence 
\[
 C_2\int_0^{2\pi}\left| \frac{\log k}{k}\right| d\theta
 \leq
 2\pi C_2 \max\left\{\frac{1}{e},\left| \frac{\log K}{K} \right| \right\}
\,.
\]
Set $G_0 = 
  \max\left\{\frac{1}{e},\left| \frac{\log K}{K} \right| \right\}$.
We observe
\[
\int_{0}^{2\pi}|\log k|d\theta
 \leq G_0\int_0^{2\pi}kd\theta\,.
\]
Thus
\[
\int_0^{2\pi}(\log k)\frac{\partial}{\partial \theta} \langle D_TV,N\rangle d\theta
 \leq 2\pi C_2 G_0
	+ 2C_1G_0\int_0^{2\pi} k\,d\theta
\,.
\]
Combining the above we find
\begin{equation}
\label{EQmidest}
  \frac{d}{dt}\int_0^{2\pi}\log k d\theta \leq 
\left(2K^* + (C_1G_0)\frac{|\sigma_2|+\sigma_2}{2\sigma_1}\right)\sigma_1\int_0^{2\pi}kd\theta
	+ G_1
\,,
\end{equation}
where
\[
 G_1 = 2\pi\sigma_1(K^*)^2+6\pi C_1+2\pi C_2\left(\frac{1}{K} + G_0\right)
\,.
\]
Recall the evolution of the length:
\[
 L'(t) = -\sigma_1\int_0^{2\pi} kd\theta-  2\pi\sigma_2-\int_0^{2\pi}\langle V,T\rangle d\theta
\,.
\]
Substitution into the estimate \eqref{EQmidest} yields
\[
 \frac{d}{dt}\int_0^{2\pi}\log k d\theta
\le 
\left(2K^* + (C_1G_0)\frac{|\sigma_2|+\sigma_2}{2\sigma_1}\right)
  \left(-L'(t)-2\pi\sigma_2 -\int_0^{2\pi}\langle V,N\rangle d\theta  \right)
	+ G_1
\,.
\]
Integrating and performing straightforward estimates finishes the proof.
\end{proof}

\subsection{Pointwise estimate}

Now we move on to estimating the pointwise maximum of the curvature in terms of entropy and other integral quantities.
First, we need the following estimate on $||k_\theta||_{L^2(d\theta)}^2$.

\begin{lem}
Let $\gamma:\S^1\times[0,T)\rightarrow\R^2$ be a solution to \eqref{flow}
satisfying the hypotheses of Theorem \ref{T:main}.
Then
\begin{align*}
\int_0^{2\pi}k^2_{\theta}\,d\theta 
 &\le e^{D_3T}\left[ 
		\int_0^{2\pi}k^2_{\theta}
		- k^2d\theta\vert_{t=0}
		+ \int_0^{2\pi}k^2d\theta
		+ \frac{4\sigma_2^2}{\sigma_1}\int_0^t\int_0^{2\pi}k^2d\theta dt
		+ D_4T
	\right]
\,,
\end{align*}
where
\[ 
   D_3 = \frac{4C_1^2}{\sigma_1K^2},\quad\text{ and } 
   D_4 = \left(\frac{72\pi C_1^2}{\sigma_1}+\frac{8\pi C_2^2}{\sigma_1K^2} \right)
\,.
\]
\label{LMconvktheta}
\end{lem}
\begin{proof}
Using integration by parts and the evolution of curvature (Lemma \ref{LMkevothetaparam}) we calculate
\begin{align*}
\frac{d}{dt} \int_0^{2\pi} (k^2-k^2_\theta)d\theta &= 2\int_0^{2\pi}(k_{\theta \theta}+k)k_td\theta
\\
	&= 2\int_0^{2\pi}\sigma_1k^2(k_{\theta \theta}+k)^2+\sigma_2k^2(k_{\theta \theta}+k)-k(k_{\theta \theta}+k)\left(  2\langle D_TV,T\rangle -\langle D_NV,N\rangle \right)
\\
	&\qquad
	- k_\theta(k_{\theta \theta}+k) \langle D_TV,N\rangle
	 + (k_{\theta \theta}+k)\langle D^2_{T,T}V, N\rangle \,d\theta
\\
	&\ge
	 2\int_0^{2\pi}\sigma_1k^2(k_{\theta \theta}+k)^2
	 - |\sigma_2|\,|k^2(k+k_{\theta \theta})|
	 - 3C_1|k(k_{\theta \theta}+k)|
\\
	&\qquad 
	 - C_1|k_\theta(k_{\theta \theta}+k)
	 - C_2|(k_{\theta \theta}+k)|\, d\theta
\\
	&\ge
	 2\int_0^{2\pi}\sigma_1k^2(k_{\theta \theta}+k)^2
	- \sigma_2|k^2(k+k_{\theta \theta})|-3C_1|k(k_{\theta \theta}+k)|
\\
	&\qquad -\frac{C_1|k_{\theta}(k_{\theta \theta}+k)||k|}{K} - \frac{C_2|(k+k_{\theta \theta})||k|}{K}
		\,d\theta
\,.
\end{align*}
In the last inequality we have used $\frac{|k|}{K}\ge 1$ in the last two
terms.
This is so that we can use the Cauchy inequality $-ab \geq -\varepsilon a^2
-\frac{b^2}{4\varepsilon}$ with $a = k(k+k_{\theta \theta})$ to absorb the
last four terms into the first.
We find
\begin{align*}
\frac{d}{dt} \int_0^{2\pi} (k^2-k^2_\theta)d\theta 
	&\ge
	 2\int_0^{2\pi} \sigma_1k^2(k+k_{\theta \theta})^2 - \varepsilon_1\sigma_2k^2(k+k_{\theta \theta})^2 -\frac{\sigma_2k^2}{4\varepsilon_1}
\\
	&\qquad 
	-3C_1\left(\varepsilon_2k^2(k+k_{\theta \theta}\right)^2+\frac{1}{4\varepsilon_2}) - \frac{C_1}{K}\left( \varepsilon_3k^2(k+k_{\theta \theta})^2+\frac{k_{\theta}^2}{4\varepsilon_3}\right)
\\
	&\qquad 
	-\frac{C_2}{K}\left( \varepsilon_4 k^2(k+k_{\theta \theta})^2+\frac{1}{4\varepsilon_4}\right)
		\,d\theta
\,.
\end{align*}
Now we choose $\varepsilon_i$ to ensure
\[
 \frac{\sigma_1}{2}
	 = \varepsilon_1 \sigma_2 +3C_1\varepsilon_2 +\frac{C_1}{K}\varepsilon_3 +\frac{C_2}{K}\varepsilon_4
\,.
\]
For instance, let us take
\[
 \varepsilon_1 = \frac{\sigma_1}{8\sigma_2},\quad
 \varepsilon_2 = \frac{\sigma_1}{24C_1},\quad
 \varepsilon_3 = \frac{\sigma_1K}{8C_1},\quad
 \varepsilon_4 = \frac{\sigma_1K}{8C_2}\,.
\]
Using this choice we find
\begin{align*}
\frac{d}{dt} \int_0^{2\pi} (k^2-k^2_\theta)d\theta
 &\ge
	\sigma_1\int_0^{2\pi} k^2(k+k_{\theta\theta})^2\,d\theta
	- 2\int_0^{2\pi}
		  \frac{2\sigma_2^2}{\sigma_1}k^2 
		+ \frac{18C_1^2}{\sigma_1}
		+ \frac{2C_1^2}{\sigma_1 K^2}k^2_{\theta}
		+ \frac{2C_2^2}{\sigma_1K^2} \, d\theta
\\
\implies
 \frac{d}{dt} \int_0^{2\pi} k^2_{\theta}d\theta 
 &\le 
	\frac{4C_1^2}{\sigma_1K^2} \int_0^{2\pi}k^2_{\theta}\,d\theta 
	+ \frac{d}{dt}\int_0^{2\pi}k^2\,d\theta
	+ \frac{4\sigma_2^2}{\sigma_1}\int_0^{2\pi}k^2\,d\theta
	+ \frac{72\pi C_1^2}{\sigma_1}+\frac{8\pi C_2^2}{\sigma_1K^2}
\,.
\end{align*}
Finally, using the Gronwall inequality we have 
\begin{align*}
\int_0^{2\pi}k^2_{\theta}\,d\theta 
 &\le
	e^{\frac{4C_1^2}{\sigma_1K^2}T}
	\left[ 
		\int_0^{2\pi}k^2_{\theta} - k^2\,d\theta\bigg{|}_{t=0}
		+ \int_0^{2\pi}k^2\,d\theta
		+ \frac{4\sigma_2^2}{\sigma_1}\int_0^t\int_0^{2\pi}k^2\,d\theta dt
		+ \left(\frac{72\pi C_1^2}{\sigma_1}
			+ \frac{8\pi C_2^2}{\sigma_1K^2} 
			\right)T
		\right]
\\
 &\le e^{D_3T}\left[ 
		\int_0^{2\pi}k^2_{\theta}
		- k^2d\theta\vert_{t=0}
		+ \int_0^{2\pi}k^2d\theta
		+ \frac{4\sigma_2^2}{\sigma_1}\int_0^t\int_0^{2\pi}k^2d\theta dt
		+ D_4T
	\right]
\,,
\end{align*}
as required.
\end{proof}

Now we derive the pointwise bound assuming control on the entropy.

\begin{lem} (Pointwise estimate.)
\label{LMpwest}
Let $\gamma:\S^1\times[0,T)\rightarrow\R^2$ be a solution to \eqref{flow}
satisfying the hypotheses of Theorem \ref{T:main}.
Suppose we have 
\[
 \int_0^{2\pi}\log k(\theta, t)d\theta \leq D_5
\,.
\]
Then 
\begin{align*}
 k_{\max}
 &\leq 
	2\exp\left(
		4D_6e^{D_3T}
			\left(
			\sqrt{2\pi}+2\sigma_2\sqrt{2\frac{\pi T}{\sigma_1}}
			\right)^2
		\right)
\\&\qquad
	+ \left(\sqrt{\int_0^{2\pi} k^2_\theta-k^2d\theta|_{t=0}}+\sqrt{D_4T}\right)
		\bigg/\left(\sqrt{2\pi}+2\sigma_2\sqrt{\frac{\pi T}{\sigma_1}}\right)
\,,
\end{align*}
where $k_{\max}$ denotes the maximum of $k(\cdot,t)$ and $D_3$, $D_4$ are as in
Lemma \ref{LMconvktheta} and $D_6 = D_5 + 2\pi|\log K|$.
\end{lem}
\begin{proof}
Recall that $k(\theta,t) > K$ for all $\theta$ and $t$. 
Fix a time $t_0$ and a positive number $\alpha$ and define
\[ 
	I_\alpha = \{\theta:\log k(\theta,t_0) \geq \alpha \}\,.
\]
Then we calculate 
\begin{align*}
D_5 \ge \int_0^{2\pi} \log k(\theta, t_0)\,d\theta 
	&= \int_{I_\alpha} \log k(\theta, t_0)\,d\theta
		+ \int_{\mathbb{S}^1\backslash I_\alpha} \log k(\theta, t_0)\,d\theta 
\\
	&\ge \alpha\mu_L(I_\alpha)
		+ \log K \mu_L(\mathbb{S}^1\backslash I_\alpha)
\,,
\end{align*}
where $\mu_L$ denotes the Lebesgue measure. Upon rearranging we have 
\begin{align*}
\alpha\mu_L(I_\alpha) &\leq D_5-\log K \mu_L(\mathbb{S}^1 \backslash I_\alpha)\\
	& \leq D_5+|\log K \mu_L(\mathbb{S}^1 \backslash I_\alpha)|\\
	&\leq D_5+ 2\pi|\log K|
\,.
\end{align*}
Set $D_6 = D_5 + 2\pi|\log K|$
 then $\mu_L(I_\alpha) \leq \frac{D_6}{\alpha}$. 
Fix $\delta = \frac{D_6}{\alpha}$.
We have $k(\theta,t_0) \leq e^{\frac{D_6}{\delta}}$ (by definition of $I_\alpha$) for all $\theta \notin
I_\alpha$. 
Then for any $\phi \in \mathbb{S}^1$ we take  $a \notin I_\alpha$
such that $\mu_L((a,\phi)) \leq \delta$ and, using Lemma \ref{LMconvktheta} in
addition to the considerations above, derive the following estimate:
\begin{align*}
k(\phi) &= k(a)+\int_a^\phi k_\theta d\theta \\
	&\le
	 e^{\frac{D_6}{\delta}}
	+ \sqrt{\delta}\left(\int_0^{2\pi}k^2_\theta\,d\theta\right)^{\frac{1}{2}}
	\\
	&\le 
	 e^{\frac{D_6}{\delta}}
	+ \sqrt{\delta}\left(e^{D_3T}\left[
		\int_0^{2\pi}(k^2_{\theta}-k^2)\,d\theta\bigg|_{t=0}
		+ \int_0^{2\pi}k^2\,d\theta
		+ \frac{4\sigma_2^2}{\sigma_1}\int_0^{t}\int_0^{2\pi}k^2\,d\theta dt
		+ D_4T\right]
	\right)^{\frac{1}{2}}
\,.
\end{align*}
Thus
\[
k_{\max}
 \le
	e^{\frac{D_6}{\delta}}
	+ e^{\frac{D_3T}{2}}\left(
		\sqrt{2\pi\delta}\, k_{\max}
		+ \sqrt{\frac{8\pi\sigma^2_2T\delta}{\sigma_1}}\,k_{\max}
		+ \sqrt{\delta\int_0^{2\pi}(k^2_{\theta}-k^2)d\theta \bigg|_{t=0}}
		+ \sqrt{D_4T\delta} 
	\right)
\,.
\]
Rearranging gives 
\[
 k_{\max}
 \le
	\frac{
		e^{\frac{D_6}{\delta}}
		+ \sqrt{\delta}e^{\frac{D_3T}{2}}\left(\sqrt{\int_0^{2\pi}(k^2_{\theta}-k^2)\,d\theta \big|_{t=0}}
			+ \sqrt{D_4T}
		\right)}
	{1
	- \sqrt\delta e^{\frac{D_3T}{2}}\left(
		\sqrt{2\pi}
		+ \sqrt{\frac{8\pi\sigma_2^2T}{\sigma_1}}
		\right)}
\,.
\]
Choosing $\delta = \frac{e^{-D_3T}}{4}\left(\sqrt{2\pi}+2\sigma_2\sqrt{\frac{2\pi T}{\sigma_1}}\right)^{-2}$
yields the result.
\end{proof}

From these estimates we conclude the proof that the flow shrinks to a point.

\begin{thm}
Let $\gamma:\S^1\times[0,T)\rightarrow\R^2$ be a solution to \eqref{flow}
satisfying the hypotheses of Theorem \ref{T:main}.
The area and length of $\gamma(\S,t)$ tend to zero as $t\rightarrow T$, and
there exists a point $\mathcal{O} \in \mathbb{R}^2$ such that 
\[
 \gamma(\mathbb{S}^1,t)\rightarrow \mathcal{O} \text{ as } t\rightarrow T
\,.
\]
\end{thm}
\begin{proof}
Suppose the area $A(\gamma(\cdot, t))\rightarrow \varepsilon
>0$ as $t\rightarrow T$.
The geometric estimate (Lemma \ref{LMgeomest}) implies a uniform estimate on
the median curvature $k^*$.
Then the integral estimate (Lemma \ref{LMintegralest}) implies uniform control
on the entropy, which then by the pointwise estimate (Lemma \ref{LMpwest})
yields a uniform estimate on the curvature. 

Thus Proposition \ref{PNbootstrap} applies and all the derivatives of
curvature are uniformly bounded.
This implies the flow is uniformly bounded in the smooth topology, enabling us to apply a local existence result to extend the flow beyond the maximal time $T$.
This is a contradiction.

Thus the area must converge to zero.
As $\gamma(\S,t)$ is convex, we are done if we rule out the possibility of
$\gamma(\cdot,t)$ converging to a straight line segment.
If this occurs, then there must exist a $\theta_0\in\S$ such that
$k(\theta_0,t)\rightarrow0$.
However this is impossible since $\min k(\cdot,t) > K > 0$ for all $t$.
Thus the only possible limiting shape is a point.
\end{proof}


\section{Speed Estimates}
\label{S:speed}

In this section we establish estimates on $F = \sigma_1k+\sigma_2 +\langle V, N \rangle$.

\begin{lem}	
\label{LMFevo}
Let $\gamma:\S^1\times[0,T)\rightarrow\R^2$ be a solution to \eqref{flow}
satisfying the hypotheses of Theorem \ref{T:main}.
Then
\[
F_t
 = \sigma_1k^2\left(F_{\theta \theta}+F\right)+F\langle D_{N}V, N \rangle +\langle V,T\rangle \langle D_TV,N\rangle
\,.
\]
\end{lem}
\begin{proof}
Recall the evolution in the angle parametrisation is given by Lemma \ref{LMkevothetaparam}; we reproduce it here for the convenience of the reader:
\begin{equation*}
k_t 
	=
	 \sigma_1k^2k_{\theta \theta}+\sigma_1k^3
	+ \sigma_2k^2
	- k(2\langle D_TV,T \rangle-\langle D_NV,N \rangle)
	- k_\theta\langle D_TV,N \rangle
	+ \langle D^2_{T,T}V, N\rangle
\,.
\end{equation*}
We calculate the first two spatial derivatives of $F$:
\begin{equation*}
F_\theta = \sigma_1k_\theta + \langle D_{\gamma_\theta}V, N \rangle - \langle V, T \rangle= \sigma_1k_\theta + \frac{1}{k}\langle D_{T}V, N \rangle - \langle V, T \rangle\,, 
\end{equation*}
and
\begin{align*}
F_{\theta \theta}
 &= 
	\sigma_1 k_{\theta \theta}
	+ \frac{1}{k} \left(\langle D^2_{T, \gamma_\theta}V,N \rangle + \langle D_NV,N \rangle-\langle D_TV,T \rangle\right)
\\&\qquad
	- \frac{k_\theta}{k^2}\langle D_TV,N \rangle 
	- \frac{1}{k}\langle D_TV,T \rangle
	- \langle V, N \rangle
\\
	&= 
	\sigma_1 k_{\theta \theta} 
	+ \frac{1}{k^2} \langle D^2_{T, T}V,N \rangle 
	- \frac{k_\theta}{k^2}\langle D_TV,N \rangle 
	+ \frac{1}{k}\langle D_NV,N \rangle 
	- \frac{2}{k}\langle D_TV,T \rangle 
	- \langle V, N \rangle.
\end{align*}
Therefore we have 
\begin{align}
F_t& = \sigma_1k^2\left(\sigma_1k_{\theta \theta} +\sigma_1k+\sigma_2 -\frac{2}{k}\langle D_TV,T\rangle+\frac{1}{k}\langle D_NV,N\rangle-\frac{k_\theta}{k^2}\langle D_TV,N \rangle +\frac{1}{k^2}\langle D^2_{T,T},N \rangle\right) +\langle V, N\rangle_t
\notag\\
&= \sigma_1k^2(F_{\theta \theta } +\langle V,N\rangle + \sigma_1k +\sigma_2)+ \langle V, N\rangle_t
\notag\\
&= \sigma_1k^2(F_{\theta \theta } +F)+ \langle V, N\rangle_t
\,.
\label{EQevolf1}
\end{align}
Next we calculate 
\begin{align*}
\langle V, N \rangle_t &= \langle D_{\gamma_t}V, N \rangle\\
	&= (\sigma_1k+\sigma_2) \langle D_{N}V, N \rangle+\langle D_{V}V, N \rangle\\
	&= (\sigma_1k+\sigma_2) \langle D_{N}V, N \rangle+\langle D_{\langle V,N \rangle N+\langle V,T \rangle T}V, N \rangle\\
	&=(\sigma_1k+\sigma_2) \langle D_{N}V, N \rangle+\langle V,N\rangle \langle D_NV,N\rangle+\langle V,T\rangle \langle D_TV,N\rangle
\,.
\end{align*}
Therefore we have
\begin{equation}
\label{EQevolf2}
 \langle V, N \rangle_t = F\langle D_{N}V, N \rangle +\langle V,T\rangle \langle D_TV,N\rangle
\,.
\end{equation}
Combining equations \eqref{EQevolf1} and \eqref{EQevolf2} gives the result.
\end{proof}

Since $\langle D_NV,N\rangle \ge -C_1$, $k>K$, we have $\sigma_1k\langle
D_NV,N\rangle \ge -\sigma_1k^2 C_1/k \ge -\sigma_1 k^2 C_1/K$.
Thus
\begin{align}
    F_t
&= 
	\sigma_1k^2(F_{\theta \theta}+F)
	+ \sigma_1k\langle D_NV,N\rangle 
	+ (\sigma_2+\langle V,N\rangle)\langle D_NV,N\rangle
	+ \langle D_TV,N\rangle\langle V,T\rangle
\notag\\
    &\geq \sigma_1k^2(F_{\theta \theta}+F-\frac{C_1}{K})-(|\sigma_2|C_1+2C_0C_1)
\,.
\label{EQft}
\end{align} 
Set $\tilde{F} = F- \frac{C_1}{K}$.
Then 
\[ 
	\tilde{F}_t \geq \sigma_1k^2(\tilde{F}_{\theta \theta}+\tilde{F})-(|\sigma_2|C_1+2C_0C_1))
	\,.
\]
We will estimate $\tilde F$.
First, we show that if $\tilde F$ and its derivative are large enough, then the
second derivative of $\tilde F$ must not be too negative.

\begin{prop}
\label{PNsturmF}
Let $\gamma:\S^1\times[0,T)\rightarrow\R^2$ be a solution to \eqref{flow}
satisfying the hypotheses of Theorem \ref{T:main}.

Set $M^2 =
\max\left\{(\tilde{F}^2+\tilde{F}^2_{\theta})(\theta,0),(\frac{C_1}{K}+|\sigma_2|+C_0)^2\right\}$.
Then $(\tilde{F}^2+\tilde{F}^2_{\theta})(\theta,t) \geq M^2$ implies
$(\tilde{F}_{\theta \theta}+\tilde{F})(\theta,t) \geq 0$.
\end{prop}
\begin{proof}
Fix $(\theta_0,t_0)$ and assume that $B =
(\tilde{F}^2+\tilde{F}_\theta^2)^{\frac{1}{2}}(\theta_0,t_0) > M$. Pick $\xi
\in (-\pi,\pi)$ such that $\tilde{F}(\theta_0,t_0) = B\cos(\xi),$ and
$\tilde{F}_\theta(\theta_0,t_0) = -B\sin(\xi)$. Consider the function 
\[
 G(\theta,t)  = \tilde{F}(\theta,t)-B\cos(\theta-\theta_0+\xi).
\]
Now $(\theta_0,t_0)$ is a double root for $G$ and we aim to show that these are
the only two roots up to periodicity. Since $\tilde{F}$ and $G$ are $2\pi $
periodic we study the interval $(\theta_0-\xi - \pi,\theta_0-\xi + \pi)$ and
show that $G$ initially has only two roots on this interval. First note that 
\[
 G(\theta_0-\xi,0) = \tilde{F}(\theta_0-\xi,0) - B < 0,
\]
and on the endpoints 
\[
 G(\theta_0-\xi\pm \pi,0) = \tilde{F}(\theta_0-\xi\pm \pi,0)+B >0.
\]
This means that there are at least two roots for $G$ initially. Now let $\theta_1 \in (\theta_0-\xi, \theta_0-\xi+\pi)$ be a root for $G$ initially, then we have that 
\[
 G_\theta(\theta_1,0) = \tilde{F}_\theta(\theta_1,0) +B\sin(\theta_1-\theta_0+\xi),.
\]
Noting that on this interval $\sin(\theta_1-\theta_0+\xi)$ is positive and also that since $\theta_1$ is a root $\tilde{F}(\theta_1,0) = B\cos(\theta_1-\theta_0+\xi)$. Then we have 
\begin{align*}
G_\theta(\theta_1,0) &> -\sqrt{B^2-\tilde{F}^2(\theta_1,0)}+B\sin(\theta_1-\theta_0+\xi)\\
	& = -\sqrt{B^2(1-\cos^2((\theta_1-\theta_0+\xi))}+B\sin(\theta_1-\theta_0+\xi)\\
	& = 0.
\end{align*}
So that $G$ is increasing at this root and in particular this implies that
there is only one root on $(\theta_0-\xi, \theta_0-\xi+\pi).$ A similar
calculation on the other interval reveals that there are at most two roots on
the entire interval $(\theta_0-\xi-\pi,\theta_0-\xi+\pi)$. Therefore, by
Sturmian theory \cite{galaktionov2004geometric} there is always at most two
roots on the interval (counted with multiplicity) hence $G(\theta_0,t_0)$ must
be the only double root. Since we have that 
\begin{align*}
 G(\theta_0-\xi\pm\pi,t_0)
 &= \tilde{F}(\theta_0-\xi\pm\pi,t_0)+B
\\
 &\ge \sigma_1k + \sigma_2 + \IP{V}{N} -\frac{C_1}{K} + \left|\frac{C_1}{K}+|\sigma_2|+C_0\right|
\\
 &\geq 0,
\end{align*}
the point $(\theta_0,t_0)$ is a minimum for $G$ therefore   
\[
 G_{\theta \theta}(\theta_0,t_0) = \tilde{F}_{\theta \theta}(\theta_0,t_0)+\tilde{F}(\theta_0,t_0)\geq 0.
\]
\end{proof}


Let us continue to estimate $F$.

\begin{prop}
Let $\gamma:\S^1\times[0,T)\rightarrow\R^2$ be a solution to \eqref{flow}
satisfying the hypotheses of Theorem \ref{T:main}.
Then
\begin{equation}\label{fgrad} 
	\sup\limits_{\theta \in [0,2\pi)} |F_\theta(\theta,t)| \leq M
	+\int_0^{2\pi}|F(\theta,t)|d\theta,
\end{equation}
\begin{equation} \label{Fmax}
	F_{\max}(t) \leq M_1\left(1+\int_0^{2\pi}|F(\theta,t)|d\theta\right),
\end{equation}
and
\begin{equation} \label{maximum}
	F_{\max}(t) \leq 2F(\theta,t)+\frac{M}{2\pi},
\end{equation}
where $M_1 = \max\left\{2\pi M, 2\pi+\frac{1}{2\pi}\right\}$ and $M$ is as in Proposition \ref{PNsturmF}.
\end{prop}
\begin{proof}
Fix a time $t_1 >0$ and  take $t_0 = \sup \{t \in [0,T):$ (\ref{fgrad}) holds  $\}.$
Suppose $t_0 < T$.
Let $t_1 = t_0+\varepsilon < T$.
If $(F^2+F^2_\theta)(\theta,t_1) \leq M^2$ then \eqref{fgrad} holds for all $\varepsilon>0$, this contradicts the definition of $t_0$.

Therefore it must be the case that $(F_\theta^2+F^2)(\theta, t_1) > M^2$ for some
$\theta \in [0,2\pi)$ and $\varepsilon>0$.
This implies that there exists $\theta_1, \theta_2 \in [0,2\pi)$ such that
$(F^2_\theta+F^2)(\theta, t_1) > M^2$ for $\theta \in (\theta_1,\theta_2)$ and
$F^2_{\theta}(\theta_2,t_1) \leq M^2$.
From Proposition \ref{PNsturmF} we thus have
\[
 (F_{\theta\theta}+F)(\theta,t_1)
 \ge 0 \text{ for } \theta \in (\theta_1,\theta_2)
\,,
\]
which gives for $\theta \in (\theta_1,\theta_2)$
\begin{align*}
 \int_\theta^{\theta_2}(F_{\theta' \theta'} +F)(\theta',t_1)\,d\theta' &\geq  0
 \\
 \int_{\theta_1}^{\theta}(F_{\theta' \theta'} +F)(\theta',t_1)\,d\theta'&\geq 0
\,.
\end{align*}
After rearranging and taking absolute values we have 
\begin{align*}
	|F_\theta(\theta,t_1)| &\leq \left\lvert F_\theta(\theta_2,t_1)+\int_\theta^{\theta_2}F(\theta,t_1)d\theta \right\rvert\\
	&\leq M+\int_0^{2\pi}|F(\theta,t_1)|d\theta.
\end{align*}
By assumption, outside of these intervals we have $F_\theta^2(\theta,t_1) \leq
M^2$.
Thus we obtain the estimate 
\[
 \sup\limits_{\theta \in [0,2\pi)} |F_\theta(\theta,t_1)| \leq M+\int_0^{2\pi}|F(\theta,t_1)|d\theta\,,
\]
that is, \eqref{fgrad}.
This is a contradiction and so $t_0 = T$.

Using the Fundamental Theorem of Calculus we have 
\[
 F(\theta,t) = \frac{1}{2\pi}\int_0^{2\pi}F(\theta,t)\,d\theta
	 + \int_a^{\theta}F_{\theta'}(\theta',t)\,d\theta'
\,.
\]
Here $a$ is any point where $F$ attains its average, that is, $F(a) = \bar{F}$.

After taking absolute values and estimating, we find
\begin{align*}
F_{\max}(t)
&\le \frac{1}{2\pi}\int_0^{2\pi}|F(\theta,t)|\,d\theta
 + \int_0^{2\pi} |F_\theta(\theta,t)|d\theta
\\
&\le
 \frac{1}{2\pi}\int_0^{2\pi}|F(\theta,t)|\,d\theta
 + 2\pi\left(M+\int_0^{2\pi}|F(\theta,t)|\,d\theta\right)
\\
&\le
 M_1\left(1+ \int_0^{2\pi} |F(\theta,t)|\,d\theta\right)\,,
\end{align*}
with $M_1 =\max\left\{2\pi M, 2\pi+\frac{1}{2\pi}\right\}$.
This proves \eqref{Fmax}.

Let $\theta^*$ be such that $F_{\max}(t) = F(\theta^*,t)$.
Set $\theta \in (0,2\pi)$ to be such that $|\theta -\theta^*| \leq \frac{1}{4\pi}$.
By the fundamental theorem of calculus we have 
\begin{align*}
F_{\max}(t) 
&=
 F(\theta,t)+\int_\theta^{\theta^*}F_\theta(\theta,t)d\theta
\\
&\le
 F(\theta,t)+|\theta^*-\theta|\sup\limits_{\theta \in [0,2\pi)}|F_\theta(\theta,t)|
\\
&\le
 F(\theta,t)+|\theta^*-\theta|\left( M+\int_0^{2\pi}|F(\theta,t)|d\theta \right)
\\
&\le
 F(\theta,t) +\frac{1}{4\pi}(M+2\pi F_{\max}(t))
\,.
\end{align*}
Rearranging yields
\[ 
	F_{\max}(t) \leq 2F(\theta,t)+\frac{M}{2\pi}\,,
\]
or \eqref{maximum}, as required.
\end{proof}


\section{Rescaling}
\label{S:rescaling}

In this section we study a rescaling of the flow, and show that it converges to
a circle.
Our method is to use an adapted monotonicity formula, which itself requires a
curvature bound.
Establishing the curvature bound forms the majority of the work in this
section.

Set the rescaling factor to be $\phi(t) = (2T-2t)^{-\frac{1}{2}}$
 and set rescaled time to $t(\hat{t}) = T(1-e^{-2\hat{t}})$.
Observe that $\phi(t(\hat{t})) =\frac{e^{\hat t}}{\sqrt{2T}}$.
The rescaled flow is denoted $\hat\gamma$ and defined by
\begin{equation}
 \hat{\gamma}(\hat{\theta}, \hat{t})
 = \phi(t(\hat{t}))(\gamma(\theta(\hat{\theta}),t(\hat{t}))-\mathcal{O})
 = \frac{e^{\hat{t}}}{\sqrt{2T}}(\gamma(\theta(\hat{\theta}),t(\hat{t}))-\mathcal{O})
\,,
\label{EQrescalingdefn}
\end{equation}
where $\mathcal{O}$ is the final point. Note that we have $\phi'(t) =
\frac{1}{\phi^3}$ and $t'(\hat{t}) = 2Te^{-2\hat{t}}$, this gives
$\frac{d\phi}{d\hat{t}} = \sqrt{2T}e^{-\hat{t}} = \phi^{-1}$. 
We calculate 
\begin{align*}
\hat{\gamma}(\hat{\theta}, \hat{t})_{\hat{t}}
 &= 
	\phi'(t(\hat{t}))t'(\hat{t})[\gamma(\theta(\hat{\theta}), t(\hat{t}))-\mathcal{O}]
	+ \phi(t(\hat{t}))t'(\hat{t})\partial_t([\gamma(\theta(\hat{\theta}), t(\hat{t}))-\mathcal{O}])
\\
&= 
	\frac{e^{\hat t}}{\sqrt{2T}}[\gamma(\theta(\hat{\theta}), t(\hat{t}))-\mathcal{O}]
	+ \phi(t(\hat{t}))((\sigma_1 k(\theta, t(\hat{t}))+\sigma_2 + V(\gamma)\cdot\nu(\hat\theta,t(\hat t))\nu(\hat{\theta}, t(\hat{t}))
\\
&=
	\hat{\gamma}
	+ \sqrt{2T}e^{-\hat{t}}(
		\sigma_1k(\theta, t(\hat{t}))
		+ \sigma_2
		+ V(\gamma)\cdot \nu(\hat{\theta}, t(\hat t)))\nu(\hat{\theta}, t(\hat t))
\\
&=
	\hat{\gamma}
	 + \left(\sigma_1\hat{k}
		+ \sqrt{2T}e^{-\hat{t}}\sigma_2
		+ \sqrt{2T}e^{-\hat{t}}V\left(\frac{\hat{\gamma}}{\phi}+\mathcal{O}\right)\cdot\hat\nu
	\right)\hat{\nu}
\,.
\end{align*} 
Summarising,
\[
 \hat{\gamma}_t(\hat{\theta}, \hat{t})
 = \hat{\gamma}
	+ \left(\sigma_1\hat{k}
	+ \sqrt{2T}e^{-\hat{t}}\sigma_2
	+ \sqrt{2T}e^{-\hat{t}}V\left(\frac{\hat{\gamma}}{\phi}+\mathcal{O}\right)\cdot\hat\nu\right)\hat{\nu}
\,.
\]

First, we use a straightforward argument to show that the area of $\hat\gamma$ converges.

\begin{lem}
Let $\gamma:\S^1\times[0,T)\rightarrow\R^2$ be a solution to \eqref{flow}
satisfying the hypotheses of Theorem \ref{T:main}.
Set $\hat\gamma$ to be the rescaled flow defined in \eqref{EQrescalingdefn}.
The area of $\hat\gamma$ converges to $\sigma_1 \pi$ as $\hat{t} \rightarrow
\infty$.
\label{LMareaconv}
\end{lem}
\begin{proof}
Recall that $A(\gamma(\cdot, t)) \rightarrow 0$ and $L(\gamma(\cdot, t))
\rightarrow 0$ as $t \rightarrow T$.
Therefore, using L'H\^opital's rule and the evolution for the area,
\begin{align*}
\lim_{\hat{t} \to \infty} \hat{A}(\hat{\gamma}(\cdot,t)) &= \lim_{\hat{t} \to \infty} \phi^2(t(\hat{t}))A(\gamma(\cdot, t))\\
& = \lim_{t \to T} \frac{A(\hat{\gamma}(\cdot,t))}{2T-2t}\\
&= \lim_{t \to T} \frac{-2\pi \omega \sigma_1 - \sigma_2L(\gamma(\cdot, t))-\int_{\gamma} \langle V(\gamma),  \nu \rangle\,ds}{-2}\\
& = \pi \omega \sigma_1 +\lim_{t \to T}\left(
	 \frac{\sigma_2 L(\gamma(\cdot, t))}{2} + \frac{1}{2}\int_{\gamma} \langle V(\gamma),  \nu \rangle\, ds\right)
\,.
\end{align*}
Since
\[
 \left|\frac{1}{2}\int_{\gamma} \langle V(\gamma),  \nu \rangle\, ds\right|
 \leq \frac{C_0}{2}L(\gamma(\cdot,t))
\,,
\]
we conclude 
\[
 \lim_{\hat{t} \to \infty} \hat{A}(\hat{\gamma}(\cdot,t)) = \pi \sigma_1
\]
as required.
\end{proof}

%

In order to obtain a bound on the rescaled curvature, we first use the speed
control from Section \ref{S:speed}.

\begin{lem}
Let $\gamma:\S^1\times[0,T)\rightarrow\R^2$ be a solution to \eqref{flow}
satisfying the hypotheses of Theorem \ref{T:main}.
The maximum of rescaled curvature satisfies 
\[
\lim_{\hat{t}\rightarrow \infty }\hat{k}_{\max}(\hat{t})e^{-\hat{t}} = 0.
\]
\label{LMrescaledcurv1}
\end{lem}
\begin{proof}
Differentiating and using $k\,ds = d\theta,$ we find
\begin{align*}
\frac{d L(t(\hat{t}))}{d \hat{t}}
 &= \left(
	- \sigma_1 \int_{\gamma}k^2ds
	- 2\pi \sigma_2
	- \int_\gamma k\langle V, \nu \rangle ds 
	\right)2Te^{-2\hat{t}}
\\
&= -2Te^{-2\hat{t}}\left(\sigma_1\int_0^{2\pi}kd\theta+\sigma_2 \int_{0}^{2\pi}d\theta+\int_{0}^{2\pi}\langle V, N \rangle d\theta \right)
\\
&= -2Te^{-2\hat{t}}\int_{0}^{2\pi}F(\theta,t(\hat{t}))d\theta
\,.
\end{align*}
Integrating this, and using the result that the length is zero at final time, yields
\begin{equation}
\label{EQspacetimeF}
 \int_0^{\infty}\int_{0}^{2\pi}F(\theta,t(\hat{t}))e^{-2\hat{t}}d\theta d\hat{t} = \frac{L(0)}{2T}.
\end{equation}
Now by the estimate \eqref{Fmax}
we have 
\begin{align*}
\int_0^\infty F_{\max}(t(\hat{t}))e^{-2\hat{t}}d\hat{t}
 &\leq 
	\int_0^\infty e^{-2\hat{t}}M_1\left(
		1 
		+ \int_0^{2\pi}|F(\theta,t(\hat{t}))|\,d\theta
	\right)d\hat{t}
\\
&\le \frac{M_1}{2}+M_1\left(\int_0^\infty \int_0^{2\pi}(\sigma_1 k + |\sigma_2| +|\langle V,N\rangle|) e^{-2\hat{t}} d\theta d\hat{t}  \right)
\\
& \leq \frac{M_1}{2}+M_1\pi C_0+ M_1\int_0^\infty \int_0^{2\pi}(\sigma_1 k + |\sigma_2|)e^{-2\hat{t}}\,d\theta\, d\hat{t}
\,.
\end{align*}
Equality \eqref{EQspacetimeF} implies
\[
 \int_0^\infty \int_0^{2\pi}(\sigma_1 k +\sigma_2)e^{-2\hat{t}}d\theta d\hat{t} = \frac{L(0)}{2T} - \int_0^\infty \int_0^{2\pi}\langle V, N \rangle e^{-2\hat{t}}d\theta d\hat{t} \leq \frac{L(0)}{2T} + C_0 \pi
\,.
\]
Therefore we obtain the estimate
\[
 \int_0^\infty F_{\max}(t(\hat{t}))e^{-2\hat{t}}d\hat{t} \leq M_1\left(\frac{1}{2} +\frac{L(0)}{2}+2\pi C_0 \right).
\]
%

Observe that $F_{\max}(t(\hat t))\rightarrow\infty$ and also by \eqref{maximum}
$F(\theta,t(\hat t))\rightarrow\infty$ for all $\theta$.
Let $\hat t_0\in(0,\infty)$ be such that $F(\theta,t(\hat t)) > M$ for all
$\hat t > \hat t_0$.
Then by Proposition \ref{PNsturmF} $(F_{\theta\theta} + F)(\theta,t(\hat t))
\ge 0$ for all $\hat t > \hat t_0$.
Set
$Q = F + t(|\sigma_2|C_1 + 2C_0C_1)$; then
$Q_t = F_t + (|\sigma_2|C_1 + 2C_0C_1) \ge 0$ and $Q(\theta,\cdot)$ is a
monotone function.
Furthermore
\begin{align*}
 \int_0^\infty Q_{\max}(t(\hat{t}))e^{-2\hat{t}}\,d\hat{t}
 &= \int_0^\infty F_{\max}(t(\hat{t}))e^{-2\hat{t}}\,d\hat{t}
	+ (|\sigma_2|C_1 + 2C_0C_1)\int_0^\infty t(\hat t) e^{-2\hat{t}}\, d\hat{t}
\\
 &= \int_0^\infty F_{\max}(t(\hat{t}))e^{-2\hat{t}}d\hat{t}
	+ T(|\sigma_2|C_1 + 2C_0C_1)\int_0^\infty e^{-2\hat{t}}-e^{-4\hat{t}}\, d\hat{t}
\\&<\infty
\,.
\end{align*}
Therefore $Q_{\max}(t(\hat{t}))e^{-2\hat{t}}\rightarrow0$ as $\hat
t\rightarrow\infty$ along the full sequence, and so also must the same be true
for $F_{\max}(t(\hat{t}))e^{-2\hat{t}}$.

Using \eqref{maximum}, we see that for any sequence $\theta_i\in\S$ (not
necessarily convergent), and any sequence $\hat t_i\rightarrow\infty$,
$F(\theta_i,t(\hat{t}))e^{-2\hat{t}}\rightarrow0$.
Let $\theta_i$ be such that $k_{\max}(t(\hat t_i)) = k(\theta_i)$.
Since $|(F-\sigma_1k)(\theta,t(\hat t))| \le |\sigma_2| + C_0$,
we find that $k_{\max}(t(\hat{t}))e^{-2\hat{t}}\rightarrow0$ as $\hat
t\rightarrow\infty$.
In view of the rescaling, $\hat k_{\max} = \sqrt{2T} e^{-\hat t}k_{\max}$, and
the conclusion follows.
\end{proof}

The proof above in fact establishes also the following estimate, which will be
useful in the proof of the entropy estimate.

\begin{cor}
Let $\gamma:\S^1\times[0,T)\rightarrow\R^2$ be a solution to \eqref{flow}
satisfying the hypotheses of Theorem \ref{T:main}.
Then
\[
\int_0^\infty \hat{k}_{\max}(\hat{t})e^{-\hat{t}}\,d\hat t < \infty
\,.
\]
\label{CYrescaledcurvintegral}
\end{cor}

\begin{rmk}
It is simple to see the integrated form of Corollary \ref{CYrescaledcurvintegral}, that is,
\[
\int_0^\infty \int_0^{2\pi} \hat k(\hat t)e^{-\hat t}\,d\theta\,d\hat t < \infty\,.
\]
It follows from
\[
 \int_0^\infty \int_0^{2\pi}(\sigma_1 k +\sigma_2)e^{-2\hat{t}}d\theta d\hat{t} \leq \frac{L(0)}{2T} + C_0 \pi
\,.
\]
\end{rmk}

Now we must derive the evolution of $\hat k$.

\begin{lem}
\label{LMkhatevo}
Let $\gamma:\S^1\times[0,T)\rightarrow\R^2$ be a solution to \eqref{flow}
satisfying the hypotheses of Theorem \ref{T:main}.
The evolution of the rescaled curvature is 
\begin{align*}
\hat{k}_{\hat{t}} &=
	- \hat{k} 
	+ \sigma_1 \hat{k}^2\hat{k}_{\theta \theta}+\sigma_1 \hat{k}^3
	+ \sqrt{2T}\sigma_2e^{-\hat{t}}\hat{k}^2
	- 2Te^{-2\hat{t}}\hat{k}( 2\langle D_T\hat{V},T \rangle-\langle D_N\hat{V},N\rangle) 
\\ &\qquad
	- 2Te^{-2\hat{t}}\hat{k}_{\theta} \langle D_T\hat{V} ,N \rangle 
	+ 2T\sqrt{2T}e^{-3\hat{t}}\langle D^2_{T,T}\hat{V}, N\rangle
\,.
\end{align*}
Here we use the notation $\hat{V} = V\left( \frac{\hat{\gamma}}{{\phi}} +\mathcal{O} \right)$.
\end{lem}
\begin{proof}
From the definition of the rescaling we have the relations
\begin{align*}
\phi(t(\hat{t})) &= \frac{1}{\sqrt{2T}}e^{\hat{t}}\\
\frac{\partial t}{\partial \hat{t}} &= 2Te^{-2\hat{t}} = \frac{1}{\phi(t(\hat{t}))^2}
\,.
\end{align*}
Using the chain rule we calculate
\begin{align*}
\hat{k}_{\hat{t}}
	= \left(
		\frac{k(t(\hat{t}))}{\phi(t(\hat{t}))} 
	\right)_{\hat{t}} 
	&= 
	\frac{\partial}{\partial \hat t}\left(\frac{1}{\phi}\right) k(t(\hat{t}))
	+ \frac{1}{\phi} \frac{\partial k}{\partial t} \frac{\partial t}{\partial \hat{t}}
\\
	&= 
	-\phi^{-1}{k} 
	+ \phi^{-3} k_t
	 = 
	- \hat{k} 
	+ \phi^{-3} k_t
\,.
\end{align*}
Here we used $\frac{k}{\phi} = \hat{k}$.  (Note also that
$\frac{k_\theta}{\phi} = \hat{k}_\theta, \frac{k_{\theta \theta}}{\phi} =
\hat{k}_{\theta \theta}$.)
Therefore
\begin{align*}
\hat{k}_{\hat{t}}
 &= 
	- \hat{k} 
	+ \frac{1}{\phi^3}\left(\sigma_1k^2k_{\theta \theta}
	+ \sigma_1k^3
	+ \sigma_2k^2
	- k(2\langle D_T V, T \rangle-\langle D_N V, N \rangle)
	- k_\theta \langle D_TV , N \rangle +\langle D^2_{T,T}V, N \rangle\right)
\\
 &= 
	- \hat{k} 
	+ \sigma_1 \hat{k}^2\hat{k}_{\theta \theta}+\sigma_1 \hat{k}^3
	+ \sqrt{2T}\sigma_2e^{-\hat{t}}\hat{k}^2
	- 2Te^{-2\hat{t}}\hat{k}( 2\langle D_T\hat{V},T \rangle-\langle D_N\hat{V},N\rangle) 
\\ &\qquad
	- 2Te^{-2\hat{t}}\hat{k}_{\theta} \langle D_T\hat{V} ,N \rangle +2T\sqrt{2T}e^{-3\hat{t}}\langle D^2_{T,T}\hat{V}, N\rangle
\,,
\end{align*}
as required.	
\end{proof}

Thus the entropy along the rescaling evolves according to
\[
 \hat{E}(\hat{t})_{\hat{t}} = \frac{d}{d\hat{t}} \frac{1}{2\pi} \int_0^{2\pi} \log \hat{k} d\theta = \frac{1}{2\pi} \int_0^{2\pi} \frac{\hat{k}_{\hat{t} }}{\hat{k}}d\theta
\,.
\]
Using Lemma \ref{LMkhatevo} we find
\begin{align*}
2\pi\hat{E}(\hat{t})_{\hat{t}} &= \int_0^{2\pi} -1 +\sigma_1 \hat{k}(\hat{k}_{\theta \theta}+ \hat{k})+\sqrt{2T}\sigma_2e^{-\hat{t}}\hat{k} - 2Te^{-2\hat{t}} (2\langle D_T\hat{V},T \rangle-\langle D_T\hat{V} ,N \rangle)\\
	&\quad
	  -2Te^{-2\hat{t}}\frac{\hat{k}_{\theta}}{\hat{k}} \langle D_T\hat{V} ,N \rangle +\frac{2T\sqrt{2T}e^{-3\hat{t} }}{\hat{k}}\langle D^2_{T,T}\hat{V}, N\rangle d\theta
\\
	& = \int_0^{2\pi} u\, d\theta
	 + \int_0^{2\pi} 3\sqrt{2T}e^{-\hat{t}}\sigma_2\hat{k}   - 2Te^{-2\hat{t}}(2 \langle D_T\hat{V},T \rangle-\langle D_T\hat{V} ,N \rangle) -2Te^{-2\hat{t}}\frac{\hat{k}_{\theta}}{\hat{k}} \langle D_T\hat{V} ,N \rangle\\
	&\quad 
	+\frac{2T\sqrt{2T}e^{-3\hat{t} }}{\hat{k}}\langle D^2_{T,T}\hat{V}, N\rangle d\theta
\,.
\end{align*}
Here we use the substitution $u := -1+\sigma_1\hat{k}(\hat{k}_{\theta \theta}+ \hat{k})-2\sqrt{2T}\sigma_2e^{-\hat{t}}\hat{k}$.

Our strategy now is to show that $f = \int_0^{2\pi}u\,d\theta$ is an integrable function in time.
This will allow an entropy estimate which eventuates in a curvature bound.
Note that we restart the count of the constants notated by $D$.

\begin{lem}(Rescaled Gradient Estimate)
There are constants $D_1,D_2, D_3$ depending only on $C_0$, $C_1$, $\sigma_2$, $K$, $\max k(\cdot,0)$, $\max k_\theta(\cdot,0)$ so that the rescaled derivative of curvature satisfies
\begin{equation*}
	\sigma_1^2 \hat{k}_\theta^2 \leq 2TD_1^2+D_12\pi\sqrt{2T}+ (D_1\sqrt{2T}+2\pi) \int_0^{2\pi} \sigma_1^2\hat{k}^2d\theta = D_2+D_3\int_0^{2\pi} \sigma_1^2\hat{k}^2d\theta
\,.
\end{equation*}
\end{lem}
\begin{proof}
Recall \eqref{fgrad}:
\[
 \sup\limits_{\theta \in [0,2\pi)} |F_\theta(\theta,t)|
 \leq M+\int_0^{2\pi}|F(\theta,t)|d\theta
\,.
\]
Since $F_\theta(\theta,t) = \sigma_1k_\theta+ \frac{\langle D_TV,N \rangle}{k} -\langle V,N\rangle$,  by the  triangle inequality
\[
 |\sigma_1 k_\theta| \leq  \sup\limits_{\theta \in [0,2\pi)}
|F_\theta(\theta,t)|+\left| \frac{\langle D_TV,N \rangle}{k}\right|
+\left|\langle V,N\rangle \right|.
\]
Therefore, using the hypotheses on $V$ \eqref{EQvhyps} (and confinement results
as appropriate), as well as Lemma \ref{LMkconvex}, we find
\[
 |\sigma_1k_{\theta}| \leq  \frac{C_1}{K}+C_0 +  M+\int_0^{2\pi}|\sigma_1 k+\sigma_2+ \langle V, N\rangle |d\theta,
\]
which implies 
\[
 |\sigma_1k_\theta| \leq D_1 +\int_0^{2\pi}\sigma_1 kd\theta
\,.
\]
Here we have
\[ D_1 = M+\frac{C_1}{K}+(2\pi+1)C_0+2\pi \sigma_2.
\]
Now we multiply through by the reciprocal of the rescaling factor to get the rescaled curvatures
\[
	|\sigma_1\hat k_\theta|
 = |\sqrt{2T}e^{-\hat{t}} \sigma_1k_\theta|
 \leq D_1\sqrt{2T}e^{-\hat{t}}+ \int_0^{2\pi}\sqrt{2T}e^{-\hat{t}}\sigma_1 kd\theta
\,.
\]
Squaring gives
\[
 \sigma_1^2 \hat{k}_\theta^2 \leq   2TD^2_1e^{-2\hat{t}}+2D_1e^{-\hat{t}}\sqrt{2T}\int_0^{2\pi}\sigma_1\hat{k}d\theta+\left(\int_0^{2\pi}\sigma_1\hat{k}d\theta\right)^2
\,.
\]
Note that the H\"older inequality implies
\[
 \left(\int_0^{2\pi}\sigma_1\hat{k}d\theta\right)^2 \leq 2\pi \int_0^{2\pi}\sigma_1^2\hat{k}^2d\theta
\,.
\]
Applying this and then using the Cauchy inequality gives
\[
 \sigma_1^2\hat{k}_\theta^2
 \leq 2TD_1^2+2D_1\sqrt{2T}\left( \int_0^{2\pi}\frac{1}{2}d\theta +\int_0^{2\pi}\frac{\sigma_1^2\hat{k}^2}{2}d\theta\right)+2\pi\int_0^{2\pi}\sigma_1^2\hat{k}^2d\theta
\,.
\]
Thus we obtain the gradient estimate
\begin{equation} \label{gradest}
	\sigma_1^2 \hat{k}_\theta^2 \leq 2TD_1^2+D_12\pi\sqrt{2T}+ (D_1\sqrt{2T}+2\pi) \int_0^{2\pi} \sigma_1^2\hat{k}^2d\theta = D_2+D_3\int_0^{2\pi} \sigma_1^2\hat{k}^2d\theta.
\end{equation}
\end{proof}

Although we won't need this improvement here, we note that by waiting, the bound above can be improved to
\[
	\sigma_1^2 \hat{k}_\theta^2 \leq \varepsilon + (2\pi+\delta)\int_0^{2\pi} \sigma_1^2\hat{k}^2d\theta\,,
\]
for $\hat t > \hat t_0$, where $\varepsilon(\hat t_0),\delta(\hat t_0) \rightarrow 0$ as $\hat t_0\rightarrow\infty$.

\begin{prop}
Let $u(\theta,\hat t) = -1+\sigma_1\hat{k}(\hat{k}_{\theta \theta}+\hat{k}) -2\sqrt{2T}\sigma_2e^{-\hat{t}}\hat{k}$.
Then
\[
	f(\hat t) = \int_0^{2\pi}u\,d\theta
	< D(e^{-2\hat t} + e^{-\hat t}\hat k + e^{-2\hat t}\hat k^2) 
	\rightarrow 0\,,\qquad\text{ as $t\rightarrow\infty$,}
\]
where $D$ is a polynomial in 
 $T$, $C_1$, $C_2$, $\sigma_1$, $\sigma_2$ and $1/K$.
\label{PNfest}
\end{prop}
\begin{proof}
Firstly, from integration by parts we have 
\begin{align*}
f &= \int_0^{2\pi}ud\theta = \int_0^{2\pi} -1+\sigma_1\hat{k}(\hat{k}_{\theta \theta}+\hat{k}) - 2\sqrt{2T}\sigma_2e^{-\hat{t}}\hat{k}d\theta\\
	&= \int_0^{2\pi} -1-\sigma_1\hat{k}^2_{\theta}+\sigma_1\hat{k}^2-2\sqrt{2T}\sigma_2e^{-\hat{t}}{\hat{k}}d\theta.
\end{align*}
Therefore the derivative of $f(\hat{t})$ is 
\begin{align*}
f'(\hat{t}) &= \int_0^{2\pi}-2\sigma_1\hat{k}_{{\theta}}\hat{k}_{{\theta}\hat{t}}+2\sigma_1 \hat{k}\hat{k}_{\hat{t}}+2\sqrt{2T}e^{-\hat{t}}\sigma_2\hat{k}-2\sqrt{2T}e^{-\hat{t}}\sigma_2\hat{k}_{\hat{t}}d{\theta}\\
	&= 2\int_0^{2\pi} \left[\sigma_1(\hat{k}_{{\theta} {\theta}}+\hat{k})-\sqrt{2T}e^{-\hat{t}}\sigma_2  \right]\hat{k}_{\hat{t}}+\sqrt{2T}e^{-\hat{t}}\sigma_2\hat{k} d\theta\\
	& = 2\int_0^{2\pi} \left[\sigma_1(\hat{k}_{{\theta} {\theta}}+\hat{k})-\sqrt{2T}e^{-\hat{t}}\sigma_2  \right]\\
	&\bigg[ -\hat{k} +\sigma_1 \hat{k}^2(\hat{k}_{\theta \theta}+ \hat{k})+\sqrt{2T}\sigma_2e^{-\hat{t}}\hat{k}^2 - 2Te^{-2\hat{t}}\hat{k}( 2\langle D_T\hat{V},T \rangle-\langle D_T\hat{V} ,N \rangle)\\
	& -2Te^{-2\hat{t}}\hat{k}_{\theta} \langle D_T\hat{V} ,N \rangle +2T\sqrt{2T}e^{-3\hat{t}}\langle D^2_{T,T}\hat{V}, N\rangle \bigg]+\sqrt{2T}\sigma_2e^{-\hat{t}}\hat k d\theta.
\end{align*}
This gives 
\begin{align*}
f'(\hat{t}) &= 
	2\int_0^{2\pi} -\sigma_1\hat{k}(\hat{k}_{\theta \theta}+\hat{k})
		 + \sqrt{2T}\sigma_2e^{-\hat{t}}\hat{k} 
		 + \sigma_1^2\hat{k}^2(\hat{k}_{\theta \theta}+\hat{k})^2
		 - \sqrt{2T}e^{-\hat{t}}\sigma_1\sigma_2\hat{k}^2(\hat{k}_{\theta \theta}+\hat{k})
\\
&\qquad		 + \sqrt{2T}\sigma_1\sigma_2e^{-\hat{t}}\hat{k}^2(\hat{k}_{\theta \theta}+\hat{k})
		 - 2T\sigma_2^2e^{-2\hat{t}}\hat{k}^2
		 - 2Te^{-2\hat{t}}\sigma_1\hat{k}(\hat{k}_{\theta \theta}+\hat{k})(2\langle D_T\hat{V}, T\rangle-\langle D_T\hat{V} ,N \rangle)
\\
&\qquad		 + 2T\sqrt{2T}e^{-3\hat{t}}\sigma_2\hat{k}(2 \langle D_T\hat{V},T\rangle-\langle D_T\hat{V} ,N \rangle)-2Te^{-2\hat{t}}\sigma_1\hat{k}_{\theta} (\hat{k}_{\theta \theta}+\hat{k}))\langle D_T\hat{V}, N\rangle
\\
&\qquad		 + 2T\sqrt{2T}\sigma_2e^{-3\hat{t}}\hat{k}_{\theta} \langle D_T\hat{V}, N\rangle +2T\sqrt{2T}e^{-3\hat{t}}\sigma_1(\hat{k}_{\theta \theta}+\hat{k})\langle D^2_{T,T}\hat{V},N\rangle 
\\
&\qquad		 - 4T^2e^{-4\hat{t}}\sigma_2\langle D^2_{T,T}\hat{V},N\rangle +\sqrt{2T}\sigma_2e^{-\hat{t}}\hat{k}\, d\theta.
\end{align*}
The fourth and fifth terms in the integrand cancel.
To begin controlling this expression, we deal with terms one to four, and the last term.
Recalling the definition of $u$ we have 
\begin{align*}
u^2 &= 1+\sigma_1^2\hat{k}^2(\hat{k}_{\theta \theta}+\hat{k})^2+8Te^{-2\hat{t}}\sigma_2^2\hat{k}^2
	-2\sigma_1\hat{k}(\hat{k}_{\theta \theta}+ \hat{k})\\
	&+4\sqrt{2T}e^{-\hat{t}}\sigma_2\hat{k}-4\sqrt{2T}e^{-\hat{t}}\sigma_1\sigma_2\hat{k}^2(\hat{k}_{\theta\theta}+ \hat{k}).
\end{align*}
This implies that 
\begin{align*}
u^2+u &= -\sigma_1\hat{k}(\hat{k}_{\theta \theta}+\hat{k})+2\sqrt{2T}e^{-\hat{t}}\sigma_2\hat{k}+\sigma_1^2\hat{k}^2(\hat{k}_{\theta \theta}+\hat{k})
\\
&\quad 
	+8T\sigma_2^2e^{-2\hat{t}}\hat{k}^2-4\sqrt{2T}e^{-\hat{t}}\sigma_1\sigma_2\hat{k}^2(\hat{k}_{\theta \theta}+\hat{k})
\,.
\end{align*}
Comparing $u^2+u$ with the first 6 terms of $f'(t)$ and the last we see that 
\begin{align*}
f'(\hat{t}) &=
	2\int_0^{2\pi} u^2+u\,d\theta
	- 20T\sigma_2^2 
	\int_0^{2\pi}
		 e^{-2\hat t}\hat{k}^2
	\,d\theta
	+ 8\sqrt{2T}e^{-\hat{t}}\sigma_1\sigma_2
	  \int_0^{2\pi} \hat{k}^2(\hat{k}_{\theta \theta}+\hat{k})\,d\theta
\\
&\quad
	+ 2\int_0^{2\pi}
		 - 2Te^{-2\hat{t}}\sigma_1\hat{k}(\hat{k}_{\theta \theta}+\hat{k})(2\langle D_T\hat{V}, T\rangle-\langle D_T\hat{V} ,N \rangle)
\\
&\qquad		 + 2T\sqrt{2T}e^{-3\hat{t}}\sigma_2\hat{k}(2 \langle D_T\hat{V},T\rangle-\langle D_T\hat{V} ,N \rangle)-2Te^{-2\hat{t}}\sigma_1\hat{k}_{\theta} (\hat{k}_{\theta \theta}+\hat{k})\langle D_T\hat{V}, N\rangle
\\
&\qquad		 + 2T\sqrt{2T}\sigma_2e^{-3\hat{t}}\hat{k}_{\theta} \langle D_T\hat{V}, N\rangle +2T\sqrt{2T}e^{-3\hat{t}}\sigma_1(\hat{k}_{\theta \theta}+\hat{k})\langle D^2_{T,T}\hat{V},N\rangle 
\\
&\qquad		 - 4T^2e^{-4\hat{t}}\sigma_2\langle D^2_{T,T}\hat{V},N\rangle\, d\theta.
\end{align*}
Using the definition of $u$ and the Cauchy inequality on the third term we have
\begin{align*}
\int_0^{2\pi} &8\sqrt{2T}e^{-\hat{t}}\sigma_2\sigma_1\hat{k}^2(\hat{k}_{\theta \theta}+\hat{k})\,d\theta 
= \int_0^{2\pi}8\sqrt{2T}\sigma_2e^{-\hat{t}}\hat{k}(u+1+2\sqrt{2T}\sigma_2e^{-\hat{t}}\hat{k})\,d\theta
\\
&= 8\sqrt{2T}\sigma_2\int_0^{2\pi}e^{-\hat{t}}\hat{k}u\,d\theta
 + 8\sqrt{2T}\sigma_2\int_0^{2\pi}e^{-\hat{t}}\hat{k}(1+2\sqrt{2T}\sigma_2e^{-\hat{t}}\hat{k})\,d\theta
\\
&\ge 
	- \frac{128T}{4}\sigma_2^2\int_0^{2\pi} e^{-2\hat{t}}\hat{k}^2\,d\theta
	- \int_0^{2\pi}u^2\,d\theta
\\
&\qquad 
	+ 8\sqrt{2T}\sigma_2\int_0^{2\pi}e^{-\hat{t}}\hat{k}\,d\theta
	+ 32T\sigma_2^2\int_0^{2\pi}e^{-2\hat{t}}\hat{k}^2\,d\theta
\\
&\ge 
	- \int_0^{2\pi}u^2\,d\theta
	+ 8\sqrt{2T}\sigma_2\int_0^{2\pi}e^{-\hat{t}}\hat{k}\,d\theta
\,.
\end{align*}
Inserting this into the expression for $f'$ above, 
\begin{align*}
f'(t) &\geq
	\int_0^{2\pi} u^2\,d\theta 
	+ 2\int_0^{2\pi}u\,d\theta
	- 20T\sigma_2^2\int_0^{2\pi}e^{-2\hat{t}}\hat{k}^2\,d\theta
	+ 8\sqrt{2T}\int_0^{2\pi}\sigma_2e^{-\hat{t}}\hat{k}\,d\theta
\\
&\quad
	+ 2\int_0^{2\pi}
		 - 2Te^{-2\hat{t}}\sigma_1\hat{k}(\hat{k}_{\theta \theta}+\hat{k})(2\langle D_T\hat{V}, T\rangle-\langle D_T\hat{V} ,N \rangle)
\\
&\qquad		 + 2T\sqrt{2T}e^{-3\hat{t}}\sigma_2\hat{k}(2 \langle D_T\hat{V},T\rangle-\langle D_T\hat{V} ,N \rangle)
		- 2Te^{-2\hat{t}}\sigma_1\hat{k}_{\theta} (\hat{k}_{\theta \theta}+\hat{k}))\langle D_T\hat{V}, N\rangle
\\
&\qquad		 + 2T\sqrt{2T}\sigma_2e^{-3\hat{t}}\hat{k}_{\theta} \langle D_T\hat{V}, N\rangle 
		 + 2T\sqrt{2T}e^{-3\hat{t}}\sigma_1(\hat{k}_{\theta \theta}+\hat{k})\langle D^2_{T,T}\hat{V},N\rangle 
\\
&\qquad		 - 4T^2e^{-4\hat{t}}\sigma_2\langle D^2_{T,T}\hat{V},N\rangle\, d\theta
\\&=
	\int_0^{2\pi} u^2\,d\theta 
	+ 2\int_0^{2\pi}u\,d\theta
	- 20T\sigma_2^2\int_0^{2\pi}e^{-2\hat{t}}\hat{k}^2\,d\theta
	+ 8\sqrt{2T}\int_0^{2\pi}\sigma_2e^{-\hat{t}}\hat{k}\,d\theta
\\
&\quad
	+ 2\int_0^{2\pi}
		 T_1 + T_2 + T_3 + T_4 + T_5 + T_6
		\, d\theta
	\,.
\end{align*}
Now it remains to estimate the remaining six terms
 $T_1,\ldots,T_6$
 in the large integral.
We begin with $T_1$:
\begin{align*}
2\int_0^{2\pi}&
 - 2Te^{-2\hat{t}}\sigma_1\hat{k}(\hat{k}_{\theta \theta}+\hat{k})(2\langle D_T\hat{V}, T\rangle-\langle D_T\hat{V} ,N \rangle)
\,d\theta
\\
	&= 
	- 4T\int_0^{2\pi} e^{-2\hat{t}}(2\langle D_T\hat{V},T\rangle-\langle D_T\hat{V} ,N \rangle) (u+1+2\sqrt{2T}e^{-\hat{t}}\sigma_2\hat{k})
\\&\ge
	 - \frac{16}{4\varepsilon_1}T^2\int_0^{2\pi}e^{-4\hat{t}}(2\langle D_T\hat{V},T\rangle-\langle D_T\hat{V} ,N \rangle)^2\,d\theta
	 - \varepsilon_1\int_0^{2\pi}u^2\,d\theta
\\&\quad
	 - 4T\int_0^{2\pi}e^{-2\hat{t}}3C_1\,d\theta
	 - 8T\sqrt{2T}\int_0^{2\pi}e^{-3\hat{t}}|\sigma_2|\hat{k}(2\langle D_T\hat{V},T\rangle-\langle D_T\hat{V} ,N \rangle)\,d\theta
\\&\geq 
	 - 2\pi e^{-2\hat{t}}\left(\frac{144T^2C_1^2}{\varepsilon_1}e^{-2\hat t} + 12TC_1\right)
	 -\varepsilon_1\int_0^{2\pi}u^2\,d\theta
\\&\quad
	 - 8T\sqrt{2T}\int_0^{2\pi}e^{-3\hat{t}}|\sigma_2|\hat{k}(2\langle D_T\hat{V},T\rangle-\langle D_T\hat{V} ,N \rangle)\,d\theta
\,.
\end{align*}
The final term above is now combined with $T_2$, which we estimate together:
\begin{align*}
 - 8T\sqrt{2T}\int_0^{2\pi}&e^{-3\hat{t}}\sigma_2\hat{k}(2\langle D_T\hat{V},T\rangle-\langle D_T\hat{V} ,N \rangle)\,d\theta
 + 4T\sqrt{2T}\int_{0}^{2\pi}e^{-3\hat{t}}\sigma_2\hat{k}(2\langle D_T\hat{V},T\rangle-\langle D_T\hat{V} ,N \rangle)\, d\theta
\\
 &\geq -4T\sqrt{2T}\int_{0}^{2\pi} 3C_1|\sigma_2|e^{-2\hat{t}}(e^{-\hat{t}}\hat{k}) \, d\theta
\\
 &\geq -4T\sqrt{2T}\int_{0}^{2\pi}\left(\frac{9C_1^2\sigma_2^2e^{-4\hat{t} }}{4} + e^{-2\hat{t}}\hat{k}^2\right)\, d\theta
\\
 &= 
	-18\pi T\sqrt{2T}C_1^2\sigma_2^2e^{-4\hat{t}}
	-4T\sqrt{2T}\int_{0}^{2\pi}e^{-2\hat{t}}\hat{k}^2\, d\theta
\,.
\end{align*}
For $T_3$, we integrate by parts and estimate:
\begin{align*}
2\int_0^{2\pi}&T_3\,d\theta
= 2\int_0^{2\pi}
	- 2Te^{-2\hat{t}}\sigma_1\hat{k}_{\theta} (\hat{k}_{\theta \theta}+\hat{k}))\langle D_T\hat{V}, N\rangle
	\,d\theta
\\&=
	-4T\int_{0}^{2\pi}e^{-2\hat{t}}\sigma_1\frac{\partial}{\partial \theta} \left(\frac{1}{2} \hat{k}^2_\theta+ \frac{1}{2}\hat{k}^2\right)\langle D_T\hat{V},N\rangle
	\,d\theta
\\&=
	2T\int_{0}^{2\pi}e^{-2\hat{t}}\sigma_1( \hat{k}^2_\theta+\hat{k}^2)\frac{\partial}{\partial \theta}\langle D_T\hat{V},N\rangle
	\,d\theta
\,.
\end{align*}
We compute
\begin{align*}
\frac{\partial}{\partial \theta}\langle D_T\hat{V},N\rangle
&= \langle D^2_{T,\frac{\hat{\gamma}_\theta}{\phi} } \hat{V}, N\rangle
	+ \langle D_N\hat{V},N\rangle
	- \langle D_T\hat{V},T\rangle
\\ 
&=
	\frac{\sqrt{2T}e^{-\hat{t} }}{\hat{k}}\langle D^2_{T,T}\hat{V},N\rangle+\langle D_N\hat{V},N\rangle 
	- \langle D_T\hat{V},T\rangle
\,.
\end{align*}
Inserting into the first $T_3$ calculation and estimating, we find
\begin{align*}
2\int_0^{2\pi}&T_3\,d\theta
 =  
	2T\int_{0}^{2\pi}e^{-2\hat{t}}\sigma_1( \hat{k}^2_\theta+\hat{k}^2)
	\bigg(
		\frac{\sqrt{2T}e^{-\hat{t} }}{\hat{k}}\langle D^2_{T,T}\hat{V},N\rangle+\langle D_N\hat{V},N\rangle 
		- \langle D_T\hat{V},T\rangle
	\bigg)
	\,d\theta
\\&\geq 
	- \left(\frac{2TC_2}{K}+4TC_1 \right)
		\int_0^{2\pi}(\sigma_1e^{-2\hat{t}}(\hat{k}_\theta^2+\hat{k}^2))
	\,d\theta
\,.
\end{align*}
Integrating the estimate (\ref{gradest})  we have 
\begin{equation}
\label{EQP65est1}
 \int_0^{2\pi}\sigma_1e^{-2\hat{t}}\hat{k}^2_\theta d\theta
 \leq
	 \frac{1}{\sigma_1}\int_0^{2\pi}\left( D_2e^{-2\hat{t}} + D_3e^{-2\hat{t}}\hat{k}^2\right)d\theta
\,,
\end{equation}
thus we finally arrive at
\begin{align*}
2\int_0^{2\pi}&T_3\,d\theta
\geq
	 -\left(\frac{2TC_2}{K}+4TC_1 \right)\int_0^{2\pi}\left(\frac{D_2}{\sigma_1}e^{-2\hat{t}}+\left(\frac{D_3}{\sigma_1}+\sigma_1\right)e^{-2\hat{t}}\hat{k}^2\right)
	\,d\theta
\\
&= -2\pi\frac{D_2}{\sigma_1}\left(\frac{2TC_2}{K}+4TC_1 \right)
	e^{-2\hat{t}}
   -\left(\frac{2TC_2}{K}+4TC_1 \right)\left(\frac{D_3}{\sigma_1}+\sigma_1\right)
	\int_0^{2\pi}e^{-2\hat{t}}\hat{k}^2
	\,d\theta
\,.
\end{align*}
Let us turn to $T_4$.
This will be estimated similarly to $T_3$, using also \eqref{EQP65est1}.
The steps are:
\begin{align*}
2\int_0^{2\pi} T_4\,d\theta
&\ge
	-4T\sqrt{2T}C_1|\sigma_2|\int_0^{2\pi} e^{-3\hat{t}}\hat{k}_{\theta}
		\,d\theta
\\&\ge
 	- \int_0^{2\pi}\sigma_1e^{-2\hat{t}}\hat{k}^2_\theta d\theta 
	- 32T^3C_1^2\sigma_2^2\sigma_1^{-1}\int_0^{2\pi} e^{-2\hat{t}}
		\,d\theta
\\&\ge
	- \frac{1}{\sigma_1}\int_0^{2\pi}\left( D_2e^{-2\hat{t}} + D_3e^{-2\hat{t}}\hat{k}^2\right)d\theta
	- 64\pi T^3C_1^2\sigma_2^2\sigma_1^{-1} e^{-2\hat{t}}
\\&\ge
	- \frac{D_3}{\sigma_1}\int_0^{2\pi} e^{-2\hat{t}}\hat{k}^2
		\,d\theta
	- \bigg(
		2\pi D_2\sigma_1^{-1}
		+ 64\pi T^3C_1^2\sigma_2^2\sigma_1^{-1} 
		\bigg)
		e^{-2\hat{t}}
\,.
\end{align*}
For $T_5$, we use again $\frac{\sqrt{2T}e^{-\hat{t} }}{\hat{k}} = \frac{1}{k} \le \frac1K$ and estimate
\begin{align*}
2\int_0^{2\pi} &T_5\,d\theta
	= 2\int_0^{2\pi} 2T\sqrt{2T}e^{-3\hat{t}}\sigma_1(\hat{k}_{\theta \theta}+\hat{k})\langle D^2_{T,T}\hat{V},N\rangle
		d\theta
\\&= 
	4T\sqrt{2T}\int_0^{2\pi}e^{-3\hat{t}}\left( \frac{u}{\hat{k}}+\frac{1}{\hat{k}}+2\sqrt{2T}e^{-\hat{t}}\sigma_2\right)\langle D^2_{T,T}\hat{V},N\rangle
		\, d\theta
\\&\geq 
	-\frac{32T^3C_2^2}{4\varepsilon_2}\int_0^{2\pi}\frac{e^{-6\hat{t} }}{\hat{k}^2}d\theta -\varepsilon_2\int_0^{2\pi}u^2
		\, d\theta
\\&\quad
	+ 4T\sqrt{2T}\int_0^{2\pi}\frac{e^{-3\hat{t} }}{\hat{k}} \langle D^2_{T,T}\hat{V},N\rangle
		\, d\theta 
	+ 16T^2\int_0^{2\pi}e^{-4\hat{t}}\sigma_2 \langle D^2_{T,T}\hat{V},N\rangle
		\, d\theta
\\&\geq 
	- \frac{8T^2C_2^2}{K^2\varepsilon_2}\int_0^{2\pi}e^{-4\hat{t}}
		\,d\theta
	- \varepsilon_2\int_0^{2\pi}u^2
		\,d\theta
\\&\quad
	- \frac{4TC_2}{K}\int_0^{2\pi}e^{-3\hat{t}}
		\,d\theta
	- 32C_2\pi T^2|\sigma_2|\, e^{-4\hat{t}}
\\&=
	- \varepsilon_2\int_0^{2\pi}u^2
		\,d\theta
	- \bigg(
		\frac{16\pi T^2C_2^2}{K^2\varepsilon_2}
		+ 32C_2\pi T^2|\sigma_2|
	\bigg)
		e^{-4\hat{t}}
	- \frac{8\pi TC_2}{K}e^{-3\hat{t}}
\,.
\end{align*}
Finally the last term $T_6$ is estimated by
\[
 8T^2\int_0^{2\pi}e^{-4\hat{t}}\sigma_2\langle D^2_{T,T}\hat{V},N\rangle d\theta
	 \geq 
		- 16\pi T^2\sigma_2C_2 e^{-2\hat{t}}
\,.
\]
Now, we combine all of our estimates for $T_1, \ldots, T_6$.
The only further estimate we make is to estimate the different (better) terms $e^{-3\hat t}$, $e^{-4\hat t}$ by $e^{-2\hat t}$.
This makes the equations simpler, but is otherwise of no benefit.

Taking $\varepsilon_1 = \varepsilon_2 = \frac{1}{3}$, and estimating $\int u^2\,d\theta \ge \frac1{2\pi} \bigg(\int u\,d\theta\bigg)^2$, we have 
\begin{equation}
 f'(\hat t) 
\geq 
	\frac{1}{6\pi}
	\bigg(
		\int_0^{2\pi}
			u
		\,d\theta
	\bigg)^2
	+ 2\int_0^{2\pi}
		u
	\,d\theta 
	- D_4\int_0^{2\pi}
		\left( 
			e^{-\hat{t}}\hat{k}
			 + e^{-2\hat{t}}\hat{k}^2 
		\right)
	\,d\theta
	- D_5 e^{-2\hat{t}}
\,.
\label{EQP65est2}
\end{equation}
Here $D_4$ and $D_5$ are non-negative polynomials in $T$, $C_1$, $C_2$, $\sigma_1$, $\sigma_2$ and $1/K$.

With the estimate \eqref{EQP65est2} we may complete the proof.
First, recall that Lemma \ref{LMrescaledcurv1} gives that $e^{-\hat t}\hat k \rightarrow 0$, so, setting $\delta:[0,\infty)\rightarrow[0,\infty)$ to be
\[
	\delta(\hat t)
	=
	D_4\int_0^{2\pi}
		\left( 
			e^{-\hat{t}}\hat{k}
			 + e^{-2\hat{t}}\hat{k}^2 
		\right)
	\,d\theta
	+ D_5 e^{-2\hat{t}}
\]
we have $\delta(\hat t)\to 0$  and $\delta(\hat t)>0$ for all $\hat t$.

Then, \eqref{EQP65est2} implies
\begin{equation}
 f'(\hat t) 
\geq 
	\frac{1}{6\pi}f^2(\hat t)
	+ 2f(\hat t)
	- \delta(\hat t)
\,.
\label{EQP65est3}
\end{equation}
Let us suppose that $f(\hat t_\delta) \ge \frac12\delta(\hat t_\delta)$.
Then, from \eqref{EQP65est3}, $f$ is increasing at this time and later times.
So, for all $\hat t>\hat t_\delta$, $f(\hat t) \ge \frac12f(\hat t_\delta) > 0$.
Then we estimate
\begin{equation}
 f'(\hat t) 
\geq 
	\frac{1}{6\pi}f^2(\hat t)
\,, \quad \text{$\hat t\ge \hat t_\delta$.}
\label{EQP65est4}
\end{equation}
Integrating, we find
\[
-\frac1{f(\hat t)} + \frac1{f(\hat t_\delta)}
\ge \frac{\hat t}{6\pi}\,, 
\]
or
\[
\frac1{f(\hat t)} 
\le 
	\frac1{f(\hat t_\delta)} - \frac{\hat t}{6\pi}\,. 
\]
Recall that $f$ is defined for all values of $\hat t$.
Thus, the above gives a contradiction for $\hat t = 6\pi/f(\hat t_\delta)$.

Therefore there does not exist a $\hat t_\delta$ such that $f(\hat t_\delta)
\ge \frac12\delta(\hat t_\delta)$, which means that $f(\hat t) <
\frac12\delta(\hat t)$ for all $\hat t$, as required.

\end{proof}

\begin{cor}
We have
\[
\int_0^\infty f(\hat t)\,d\hat t
= D_0 < \infty
\]
where $D_0$ depends only on $T$, $C_1$, $C_2$, $\sigma_1$, $\sigma_2$ and $1/K$.
\label{CYfisintegrable}
\end{cor}
\begin{proof}
Observe that $e^{-2\hat t}\hat k^2 \le c e^{-\hat t}\hat k$ due to Lemma \ref{LMrescaledcurv1}.
Thus, the estimate from Proposition \ref{PNfest} combined with Corollary \ref{CYrescaledcurvintegral} implies
\begin{align*}
\int_0^\infty f(\hat t)\,d\hat t
	< D
		\int_0^\infty 
			(e^{-2\hat t} + (c+1)e^{-\hat t}\hat k)
		\,d\hat t
	< \infty
\end{align*}
as required.
\end{proof}

\begin{lem}
We have
\[
 \hat{E}(\hat{t}) \leq D_6
\,,
\]
where 
$D_6$ depends only on $T$, $C_1$, $C_2$, $\sigma_1$, $\sigma_2$, $\hat E(0)$ and $1/K$.
\end{lem}
\begin{proof}
We calculate
\begin{align}
 2\pi\hat{E}(\hat{t})_{\hat{t}} 
 &\leq 
	f(\hat t) 
	+ \int_0^{2\pi} 
		3\sqrt{2T}e^{-\hat{t}}\sigma_2\hat{k}  
		- 2Te^{-2\hat{t}} \langle D_T\hat{V},T \rangle
		- 2Te^{-2\hat{t}}\frac{\hat{k}_{\theta}}{\hat{k}} \langle D_T\hat{V},N \rangle
		+ \frac{2T\sqrt{2T}e^{-3\hat{t} }}{\hat{k}}\langle D^2_{T,T}\hat{V}, N\rangle 
	\,d\theta
\notag
\\&\leq 
	\bigg(
	f(\hat t) 
	+ 3\sigma_2\sqrt{2T}\int_0^{2\pi} 
		e^{-\hat{t}}\hat{k}
	\,d\theta
	+ 2\pi\left(2TC_1 +\frac{2 TC_2}{K}\right)e^{-\hat{t}}
	\bigg)
	- \int_0^{2\pi}
		2Te^{-2\hat{t}}\frac{\partial}{\partial \theta}(\log \hat{k}) \langle D_T\hat{V},N \rangle
	\,d\theta 
\,. 
\label{EQ68est1}
\end{align}
The bracketed terms on the right hand side are integrable on $(0,\infty)$, and so we do not need to estimate them further.

Lemma \ref{LMkconvex} and the definition of the rescaling implies the estimate
\begin{equation*}
\frac{\log \hat k}{\hat k}
\le
	(2T)^{-\frac12}e^{\hat t} 
		\frac{\log k}{k}
	+ \frac{\hat t}{\hat k} + \frac{\log 2T}{2\hat k}
\end{equation*}
and so
\begin{equation*}
\left|\frac{\log \hat k}{\hat k}\right|
\le 
	(2T)^{-\frac12}e^{\hat t} 
	\bigg(
		\max\left\{\frac{\log K}{K},\frac1e\right\}
		+ \frac{\hat t}{K} + \frac{|\log 2T|}{2K}
	\bigg)
\,.
\end{equation*}
These estimates imply respectively
\begin{equation}
\label{EQ68est2}
e^{-2\hat t}\left|\log \hat k\right|
\le C(T,K)\left(
	e^{-\hat t}\hat k
	+ e^{-2\hat t}\hat t
	+ e^{-2\hat t}
\right)
\qquad
\text{and}
\qquad
e^{-2\hat t}\left|\frac{\log \hat k}{\hat k}\right|
\le C(T,K)
\,.
\end{equation}
For the remaining term in \eqref{EQ68est1}, we break it down with integration by parts as follows:
\begin{align}
-2T\int_0^{2\pi}e^{-2\hat{t}}&\frac{\partial}{\partial \theta}(\log \hat{k}) \langle D_T\hat{V},N \rangle d\theta 
\notag
\\ &=  
	2T\int_0^{2\pi}
		e^{-2\hat{t}}\log \hat{k}\left(
			\frac{\sqrt{2T}e^{-\hat t}}{\hat k}\langle D^2_{T,T}\hat{V},N \rangle +\langle D_N\hat{V},N \rangle-\langle D_T\hat{V},T \rangle
		\right)
	\,d\theta 
\,.
\notag
\end{align}
Expanding the bracket gives three integrals.
We estimate the first by using \eqref{EQ68est2}:
\begin{align*}
2T\sqrt{2T}\int_0^{2\pi}
	e^{-3\hat{t}}\frac{\log \hat{k}}{\hat k}\langle D^2_{T,T}\hat{V},N \rangle
\,d\theta
&\le  
	2TC_2\sqrt{2T}\int_0^{2\pi}
		e^{-3\hat{t}}\left|\frac{\log \hat k}{\hat k}\right| 
	\,d\theta
\\
&\le  
	4TC_2C(T,K)\pi\sqrt{2T}
		e^{-\hat{t}}
\,.
\end{align*}
Thus this piece is integrable in $\hat t$.

We use \eqref{EQ68est2} to estimate the remaining terms:
\begin{align*}
2Te^{-2\hat{t}}\int_0^{2\pi}
	&\log \hat{k} (\langle D_N\hat{V},N \rangle-\langle D_T\hat{V},T \rangle)
	\,d\theta
\le 
	4Te^{-2\hat{t}}C_1\int_0^{2\pi} 
		|\log \hat k|
	\,d\theta
\\&\le
 	C(T,K,C_1)\left(
	e^{-\hat t}\hat k_{\max}
	+ e^{-2\hat t}\hat t
	+ e^{-2\hat t}
	\right)
\,.
\end{align*}
Each of these are integrable in $\hat t$ (using Corollary \ref{CYrescaledcurvintegral} for the first).
Thus the result follows from integrating \eqref{EQ68est1}.
\end{proof}

Finally, we give the rescaled curvature estimate.

\begin{thm}
There exists a constant  $\hat{k}_1$ depending only on 
$T$, $L_0$, $C_0$, $C_1$, $C_2$, $\sigma_1$, $\sigma_2$, $\hat E(0)$ and $K$ such that 
\[
 \hat{k}_{\max}(\hat{t}) \leq \hat{k}_1  \qquad\text{ for all }\qquad \hat{t} \in [0,\infty)
\,.
\]
\label{TMresccurvbd}
\end{thm}
\begin{proof}
Recall the speed estimate \eqref{maximum}:
At a fixed time $t$ and in a neighbourhood $\{\theta\in\S\,:\,|\theta-\theta^*| \leq \frac{1}{4\pi}\}$ about the spatial maximum of curvature $k(\theta^*,t)$ we have 
\[ 
	\sigma_1k_{\max}(t) < 2(\sigma_1k(\theta,t)+\sigma_2+\langle V,N \rangle )+C_0+\frac{M}{2\pi}
	\,.
\]
In rescaled variables, this reads
\[
\sigma_1\hat{k}_{\max}(\hat{t})
< 
	2\sigma_1\hat{k}(\theta, \hat{t})+\left(2\sigma_2+3C_0+\frac{M}{2\pi}\right)\sqrt{2T}e^{-\hat{t}}
\,.
\]
Rearranging gives
\[ 
\hat{k}(\theta, \hat{t}) \geq \frac{1}{2}\left( \hat{k}_{\max}(\hat{t})-D_7e^{-\hat{t}}\right)
\,,
\]
here the constant $D_7$ is
\[
D_7 =\frac{1}{\sigma_1}\left(3C_0+2\sigma_2+\frac{M}{2\pi}\right).
\]
Taking logarithms and integrating over the set $\left\{ |\theta-\theta^*|\leq \frac{1}{4\pi} \right\}$ implies 
\begin{align}
\int_{|\theta-\theta^*|\leq \frac{1}{4\pi}}\log\hat{k}
\,d\theta 
&\ge 
	\int_{|\theta-\theta^*|\leq \frac{1}{4\pi}}\log\left(\hat{k}_{\max}(\hat{t})-D_7e^{-\hat{t}}\right)
	\,d\theta
	-2\pi\log2
\notag\\&\ge 
	\frac1{4\pi}\log\left(\hat{k}_{\max}(\hat{t})-D_7e^{-\hat{t}}\right)
	-2\pi\log2
\,.
\label{EQkhatest1}
\end{align}
Now suppose that $\hat{k}(\hat t)\rightarrow\infty$ as $\hat t\rightarrow\infty$.
Hence there exists a monotone sequence of times $\{ \hat{t}_j\}$ with $\hat{t}_j\rightarrow \infty$ such that $\hat{k}_{\max}(\hat{t}_j) \rightarrow\infty$. 
The entropy bound implies 
\[
 D_6 \geq E(\hat{t}_j) = \frac{1}{2\pi} \int_0^{2\pi}\log \hat{k}(\theta, \hat{t}_j)\,d\theta\,.
\]
Fix a time $\hat{t}_j$ and split the space domain in the following way 
\[
[0,2\pi)
 = \left\{|\theta-\theta^*| \leq \frac{1}{4\pi} \right\} 
	\cup \left\{ \{ \hat{k} < 1  \}\setminus  \left\{ |\theta-\theta^*| \leq \frac{1}{4\pi} \right\} \right\} 
	\cup \left\{ \{ \hat{k} \geq 1  \}\setminus  \left\{ |\theta-\theta^*| \leq \frac{1}{4\pi} \right\} \right\}
\,.
\] 
We wish to estimate the entropy on each of these disjoint sets.
On the first set the estimate is \eqref{EQkhatest1}.
On the second set we have $ 0 > \hat{k}\log{\hat{k}} \geq -e^{-1} $ for $\hat{k} \in (0,1)$ and therefore
\begin{align*}
\int_{\{ \hat{k} < 1 \} \setminus \{|\theta-\theta^*| \leq \frac{1}{4\pi}\} }
	\log \hat{k}(\theta, \hat{t})
\,d\theta 
&\ge 
	\int_{\{ \hat{k} < 1 \} \setminus \{|\theta-\theta^*| \leq \frac{1}{4\pi}\}}
		\hat{k}(\hat{s})\log \hat{k}(\hat{s}, \hat{t})
	\,d\hat{s}
\\&\ge 
	\int_{\{ \hat{k} < 1 \} \setminus \{|\theta-\theta^*| \leq \frac{1}{4\pi}\}} -e^{-1}
	\,d\hat{s}
\\&\ge
	-e^{-1}\hat{L}_0
\,.
\end{align*}
Finally on the third set $\log(\hat k) \geq 0$ and so trivially 
\[
	\int_{\{ \hat{k} < 1 \} \setminus \{|\theta-\theta^*| \leq \frac{1}{4\pi}\} } \log \hat{k}(\theta, \hat{t})\,d\theta \ge 0\,.
\]
Thus
\[
D_6 \ge \frac1{2\pi} \int_0^{2\pi}\log \hat{k}(\theta, \hat{t}_j)\,d\theta
 \ge \frac1{2\pi}\bigg(
	\frac1{4\pi}\log\left(\hat{k}_{\max}(\hat{t})-D_7e^{-\hat{t}}\right)
	-2\pi\log2
	-e^{-1}\hat{L}_0
	\bigg)
\,.
\]
The right hand side $\to\infty$ and the left hand side is bounded.
This is a contradiction.
Furthermore, the above equation is an explicit estimate for $\hat k_{\max}$, which gives the dependencies of the constant $\hat k_1$.
\end{proof}

The curvature bound allows us to show that the rescaled curvature is bounded from below.

\begin{thm}
The minimum of rescaled curvature is bounded from below.
\end{thm}
\begin{proof}
First, we need a reverse isoperimetric inequality.
Lemma 3.2 from \cite{reverse} applies, to yield
\[
	\hat L(\hat t) \le 2\hat k_{\max} \hat A(\hat t)\,.
\]
Note that, as mentioned in \cite{reverse}, this also holds by using the Minkowski inequalities.
Then, using Lemma \ref{LMareaconv} and the rescaled curvature bound, Theorem \ref{TMresccurvbd}, we find
\begin{equation}
\label{EQhatlbd}
	\hat L(\hat t) \le C\,.
\end{equation}
The constant $C$ above depends on $T, L_0, C_0, C_1, C_2, \sigma_1, \sigma_2, \hat E(0), K, \max k(\cdot,0)$, and $\max k_\theta(\cdot,0)$.
In this proof, $C$ will refer to any such constant.

Let $\theta^*\in [0,2\pi]$ be such that $\hat{k}(\theta^*,t) = \hat{k}_{\min}(t)$ then by the fundamental theorem of calculus, the gradient bound and curvature bound from above, we have 
\begin{align*}
\sigma_1(\hat{k}-\hat{k}_{\min})(\theta,t) 
= \sigma_1(\hat{k}-\hat{k}_{\min})(\theta^*,t) +
	\int_{\theta^*}^{\theta}
	\sigma_1\hat{k}_{\theta}
	\,d\theta
\leq 
	|\theta-\theta^*|C
	\,.
\end{align*}

Reparametrise so that $\theta^* = 0$ (this does not change the minimum of $\hat k$).
Then we calculate
\begin{align*}
\hat L = 
	\int_{-\pi}^{\pi}
		\frac{1}{\hat{k}}
	\,d\theta 
&\geq 
	\int_{-\pi}^{\pi} 
		\frac{1}{\hat{k}_{\min}+C|\theta|}
	\,d\theta
\\&= 
	2\int_{0}^{\pi}
		\frac{1}{\hat{k}_{\min}+C\theta}
	\,d\theta
\\&= 
	\frac{2}{C}\log\left(\frac{\hat{k}_{\min}+C\pi}{\hat{k}_{\min}}\right)
\,.
\end{align*}
Upon rearranging this gives 
\[ 
	\hat{k}_{\min} \geq \frac{C\pi}{e^{\frac{C\hat L}{2}}-1}
\,.
\]
Combining with \eqref{EQhatlbd} finishes the proof.
\end{proof}

Finally we prove that $\hat\gamma$ converges to a round circle.
This is done using a Gaussian-type integral similar to Huisken's monotonicity formula \cite{huisken1990asymptotic}. 
The choice of kernel we use here follows \cite{cesaroni2011curve} (see also \cite{he2019curvature}).
While we do not obtain monotonicity of our integral, it does satisfy good enough estimates to be useful.
Our main task is to control the additional terms that arise due to the presence of the ambient field $V$.

We need the evolution of $|\hat{\gamma}_u|$, which is a straightforward calculation.

\begin{lem}
The evolution of $|\hat{\gamma}_u|$ is
\[
 |\hat{\gamma}_u|_{\hat{t}} = |\hat{\gamma}_u|\left[1 -\hat{k}(\sigma_1\hat{k} +\sqrt{2T}e^{-\hat{t}}\sigma_2) +2Te^{-2\hat{t}} \langle D_{\hat{\tau}}\hat{V},\hat{\tau} \rangle \right].
\]
\end{lem}
\begin{proof}
Calculating,
\begin{align*}
\frac{\partial}{\partial \hat{t}} |\hat{\gamma}_u|^2 
&= 2 |\hat{\gamma}_u| |\hat{\gamma}_u|_{\hat{t}} = 2\langle \hat{\gamma}_{ut}, \hat{\gamma}_{u} \rangle\\
\implies |\hat{\gamma}_u|_{\hat{t}}& = \langle \partial_s \hat{\gamma}_{\hat{t}}, \hat{\gamma}_u\rangle\\
	&=|\hat{\gamma}_u|\langle \partial_s \hat{\gamma}_{\hat{t}}, \hat{\gamma}_s\rangle
\,.
\end{align*}
Therefore  
\begin{align*}
|\hat{\gamma}_u|_{\hat{t}} 
&= 
	|\hat{\gamma}_u|\left\langle 
		\partial_s\left[
			\hat{\gamma} + (\sigma_1\hat{k}+\sqrt{2T}e^{-\hat{t}}\sigma_2)N
			+ \sqrt{2T}e^{-\hat{t}}V\left(\frac{\hat{\gamma}}{\phi} + \mathcal{O}\right)\right],
			T
		\right\rangle 
\\
&= 
	|\hat{\gamma}_u|\left[
		1 
		- \hat{k}(\sigma_1\hat{k}+\sqrt{2T}e^{-\hat{t}}\sigma_2)
		+ 2Te^{-2\hat{t}}\langle D_{T}\hat{V},T\rangle
	\right]
\,,
\end{align*}
as required.
\end{proof}

Now we calculate the evolution of a Gaussian integral.

\begin{lem}
Let 
$\rho = e^{-\frac{|\hat{\gamma}|^2}{2\sigma_1}}$ and 
$R(\hat{t}) = \int_0^{2\pi} \rho\, d\hat s$.
We have
\[ 
R'(\hat{t}) 
= 
	-\int_{0}^{2\pi}
		Q^2\rho 
	\,d\hat s 
	+ \frac{T\sigma_2^2}{2\sigma_1}e^{-2\hat{t}}\int_{0}^{2\pi}
		\rho 
	\,d\hat s 
	+ \int_{0}^{2\pi}
		\bigg( 
			2Te^{-2\hat{t}}\langle D_{T} \hat{V},T \rangle 
			- \frac{2Te^{-\hat{t} }}{\sigma_1}\langle \hat{V}, N \rangle 
		\bigg)
	\,d\hat s,
\]
where 
\[ 
Q  
= 
	\frac{  \langle \hat{\gamma}, {N} \rangle}{\sqrt{\sigma_1}} +\sqrt{\sigma_1} \hat{k} +\frac{\sqrt{2T}\sigma_2}{2\sqrt{\sigma_1}}e^{-\hat{t}}.
\]
\label{LMmonoton}
\end{lem}
\begin{proof}
Noting that $\rho_{\hat{t}} = -\frac{\langle \hat{\gamma}_{\hat{t}}, \hat{\gamma} \rangle}{\sigma_1}\rho $ we calculate 
\begin{align*}
R'(t)
 &= 
	\int_{0}^{2\pi} 
		\rho |\hat{\gamma}_u|_{\hat{t}}
	\,d\hat s
	+ \int_{0}^{2\pi} 
		\rho_{\hat{t}}
	\,d\hat s
\\&=
	\int_{0}^{2\pi}
		\rho \left( 
			1
			- \sigma_1\hat{k}^2
			- \sqrt{2T}e^{-\hat{t}}\sigma_2\hat{k} 
			- \frac{1}{\sigma_1}\Big\langle 
				\hat{\gamma},
				\hat{\gamma}
				+ \Big(
					\sigma_1\hat{k}
					+ \sigma_2\sqrt{2T}e^{-\hat{t}}
				\Big)
			N\Big\rangle
		\right) 
	\,d\hat s
\\&\qquad
	+ \int_{0}^{2\pi}
		\rho \left(
			2Te^{-2\hat{t}}\langle 
				D_{T} \hat{V}, T
			\rangle 
		- \frac{\sqrt{2T}e^{-\hat{t} }}{\sigma_1}\langle 
			\hat{V}, 
			\hat{\gamma}\rangle
		\right)
	\,d\hat s
\\&=
	\int_{0}^{2\pi}
		\left( 
			1
			- \sigma_1\hat{k}^2
			- \sqrt{2T}e^{-\hat{t}}\sigma_2\hat{k} 
			- \frac{1}{\sigma_1}|\hat\gamma|^2
			- \hat k\langle \hat\gamma, N\rangle
			- \frac{\sigma_2\sqrt{2T}e^{-\hat{t} }}{\sigma_1}\langle \hat\gamma, N\rangle
		\right) 
	\rho\,d\hat s
\\&\qquad
	+ \int_{0}^{2\pi}
		\left(
			2Te^{-2\hat{t}}\langle 
				D_{T} \hat{V}, T
			\rangle 
		- \frac{\sqrt{2T}e^{-\hat{t} }}{\sigma_1}\langle 
			\hat{V}, 
			\hat{\gamma}\rangle
		\right)
	\rho\,d\hat s.
\end{align*}
The square of $Q$ is
\[
Q^2 
= 
	\frac{ \langle \hat{\gamma}, N \rangle^2}{\sigma_1} 
	+ \sigma_1\hat{k}^2 
	+ \frac{2T\sigma_2^2}{4\sigma_1}e^{-2\hat{t}} 
	+ 2\hat{k}\langle \hat{\gamma},N\rangle 
	+ \sigma_2\sqrt{2T}\hat{k}e^{-\hat{t}}
	+ \frac{\sigma_2\sqrt{2T}e^{-\hat{t}}\langle \hat{\gamma}, N\rangle}{\sigma_1}
\,.
\]
Inserting this into $R'$ yields
\begin{align*}
R'(\hat{t}) 
&=
	- \int_{0}^{2\pi} 
		Q^2\rho 
	\,d\hat s 
	+ \int_{0}^{2\pi}
		\left(
			1
			- \frac{|\hat{\gamma}|^2}{\sigma_1}
			+ \hat{k} \langle \hat{\gamma}, N\rangle 
			+ \frac{\langle \hat{\gamma}, N\rangle^2}{\sigma_1} 
			+ \frac{T\sigma_2^2e^{-2\hat{t} }}{2\sigma_1}
		\right)\rho
	\,d\hat s
\\&\qquad
	+ \int_{0}^{2\pi}
		\left(
			2Te^{-2\hat{t}}\langle D_{T} \hat{V}, T \rangle 
			- \frac{\sqrt{2T}e^{-\hat{t} }}{\sigma_1} \langle \hat{V}, \hat{\gamma}\rangle   
		\right) 
	\rho\,d\hat s.
\end{align*}
Since
\[
\langle \hat{\gamma}, T\rangle_{\hat s}
=
	1 
	+ \hat{k}\langle \hat{\gamma}, N\rangle,
\]
we have
\[ 
\int_{0}^{2\pi}
	(1+\hat{k}\langle \hat{\gamma}, N\rangle)
\rho\,d\hat s 
= 
	\int_{0}^{2\pi}
		\frac{1}{\sigma_1} \langle \hat{\gamma},T\rangle^2
	\rho\,d\hat s
\,.
\]
Note that we also used $\partial_{\hat s} \rho = - \frac{\langle \hat{\gamma}, T\rangle}{\sigma_1}\rho$.
This cancels with other terms nicely, since 
\[
-\frac{1}{\sigma_1}|\hat{\gamma}|^2
= 
	-\frac{1}{\sigma_1}\left(  
		\langle \hat{\gamma}, N\rangle^2
		+ \langle \hat{\gamma}, T\rangle^2 
	\right)
\,.
\] 
Therefore we conclude that 
\[ 
R'(\hat{t}) 
= 
	-\int_{0}^{2\pi} 
		Q^2
	\rho\,d\hat s
	+ \int_{0}^{2\pi}\frac{T\sigma_2^2e^{-2\hat{t} }}{2\sigma_1}
	\rho\,d\hat s
	+ \int_{0}^{2\pi}
		\left(2Te^{-2\hat{t}}\langle D_{T}\hat{V}, T\rangle - \frac{\sqrt{2T}e^{-\hat{t} }}{\sigma_1} \langle \hat{V}, \hat{\gamma}\rangle
		\right) 
	\rho\,d\hat s
\,,
\]
as required.
\end{proof}

We can now finish our argument following Huisken \cite{huisken1990asymptotic}.
Integrating Lemma \ref{LMmonoton} we find 
\begin{align*}
 R(0)-R(\hat{t}_j) 
+ \frac{T\sigma_2^2}{2\sigma_1}\int_0^{\hat{t}_j}
	e^{-2\hat{t}_j}\int_{0}^{2\pi} 
		\rho 
	\,d\hat s
\,d\hat{t}
&+ \int_0^{\hat{t}}
	\int_{0}^{2\pi} 
		\left(
			2Te^{-2\hat{t}}\langle D_{\hat{\tau}} \hat{V}, T\rangle 
			- \frac{\sqrt{2T}e^{-\hat{t} }}{\sigma_1} \langle \hat{V}, \hat{\gamma}\rangle   
		\right) 
	\rho\,d\hat s
\,d\hat{t} 
\\&= 
	\int_0^{\hat{t}_j}
		Q^2
	\rho\,d\hat s 
\,d\hat{t}
\,.
\end{align*}
Along any sequence of times $\hat t_j$, we have that
\[
 \lim_{j\rightarrow \infty} \int_0^{\hat{t}_j} \int_{0}^{2\pi}Q^2\rho d\hat s d\hat{t} 
\]
is bounded.
This implies there is a subsequence $\{ \hat{t}_i\}$ along which $\int_{0}^{2\pi}Q^2\rho d\hat s \rightarrow 0$.
As we have already established uniform rescaled curvature and length estimates, and we have the confinement conditions that enable control on $\hat V$ and all its derivatives, we obtain bounds on all derivatives of rescaled curvature.
Following now the argument in \cite[Proposition 3.4]{huisken1990asymptotic} gives smooth convergence to a solution of 
\[
 \langle \hat\gamma ,N \rangle  = -\sigma_1\hat k
\,.
\]
Applying finally Abresch-Langer \cite{abresch1986normalized}, and embeddedness of $\hat\gamma$, we conclude that $\hat\gamma$ is a standard round circle.


\appendix

\section{Ambient Killing Force Fields}

We briefly discuss the case where the ambient vector field $V$ is Killing, i.e.
it is generated by the three Killing vector fields in the plane $\xi_1(x,y) = (1,0)$,
$\xi_2(x,y) = (0,1)$ and $\xi_3(x,y) = (-y,x)$.
In the below, a \emph{wound healing flow} (referring to \cite{he2019curvature})
is a flow of the form \eqref{flow} with $V=0$.

\begin{lem}
If the vector field $V$ is a Killing field then there is an associated wound
healing flow equivalent to \eqref{flow} by rigid motion.
\end{lem}
\begin{proof}
For the wound healing flow we have 
\[
 k_t = \sigma_1k^2k_{\theta \theta}+\sigma_1k^3+\sigma_2k^2\,.
\]
Since $V$ is Killing there exist $a,b,c$ such that $V(x,y) = a\xi_3(x,y) +
b\xi_1(x,y) + c\xi_2(x,y) = a(y,-x)+(b,c)$, where $a,b,c \in \mathbb{R}$.
Letting $T = (\alpha,\beta), N = (-\beta, \alpha)$ we find 
\[
 D_TV = \begin{bmatrix} 0&a\\
	-a&0
\end{bmatrix} \begin{bmatrix} \alpha\\
	\beta
\end{bmatrix}  = \begin{bmatrix} a\beta\\
	-a\alpha
\end{bmatrix}, \implies \langle D_TV,N\rangle = -a
\]
and $\langle D_TV,T\rangle = \langle D_NV,N\rangle = 0$ as well as $D^2V = 0$.
Let $\bar{k}$ be the curvature of a curve evolving by the flow \eqref{flow}.
The curvature evolves via (see Lemma \ref{LMkevothetaparam})
\[
 \bar{k}_t = \sigma_1\bar{k}^2\bar{k}_{\theta \theta}+\sigma_1\bar{k}^3+\sigma_2\bar{k}^2+a\bar{k}_{\theta}\,.
\]
Observe that $\bar{k}$ and $k(\theta+at,t)$ satisfy same evolution equation,
and by uniqueness, must be equal.
Therefore from the fundamental theorem of curves there must exist an isometry
between the two curves they represent.
\end{proof}

In fact, it is not hard to prove what this isometry is.
Let $(x(t),y(t))$ be an integral curve of $V$ starting from the origin, that is,
\begin{align*}
	x'(t) &= -y(t)a+b\\
	y'(t)& = ax(t)+c\\
	(x(0),y(0)) &= (0,0)\,.
\end{align*}
The solution, which can be found using matrix exponentials and Duhamel's formula, is given by
\[ \begin{pmatrix}
	x(t)\\
	y(t)
\end{pmatrix} =
\frac{1}{a}
\begin{pmatrix}
	b\sin(at)+c\cos(at)-c\\
	c\sin(at)-b\cos(at)+b
\end{pmatrix}\,.
\]

Below we give a more formal statement.

\begin{lem}\label{killingrigid}
Let $V(x,y) = a\xi_3(x,y) + b\xi_1(x,y) + c\xi_2(x,y) = a(y,-x)+(b,c)$, where
$a,b,c \in \mathbb{R}$ and set $R_\theta$ to be anticlockwise rotation by angle
$\theta$.

Suppose $\gamma$ satisfies the wound healing flow.
Then the family of curves 
\[
 \hat{\gamma} = R_{at}\gamma + (x(t),y(t))
\]
satisfies
\begin{align*}
	\hat{\gamma}_t &= (\sigma_1\hat{k}+\sigma_2)\hat{\nu}+V(\hat{\gamma})\\
	\hat{\gamma}(\cdot,0) &= \gamma_0\,.
\end{align*}
\end{lem}
\begin{proof}
We calculate 
\begin{equation}
	\hat{\gamma}_t =R_{at}\gamma_t+ \frac{d}{dt}(R_{at})\gamma  +(x',y').
\end{equation}
Differentiating the rotation $R_{at}$ we find 
\[
 \frac{d}{dt}\begin{bmatrix} \cos(at) &-\sin(at)\\
	\sin(at) & \cos(at)
\end{bmatrix} = 
a\begin{bmatrix}
	-\sin(at) & -\cos(at)\\
	\cos(at) &-\sin(at)	
\end{bmatrix}.
\]
By considering the facts that $\cos(\frac{\pi}{2}+at) = -\sin(at)$ and
$\sin(\frac{\pi}{2}+at)  = \cos(at)$ it is clear that 
\[
 \frac{d}{dt}R_{at} = aR_{\frac{\pi}{2}+at}\,.
\]
We also calculate 
\begin{align*}
\hat{\gamma}_u &= R_{at}\gamma_u,\\
\implies \hat{\tau} &= R_{at}\tau,\\
\implies \hat{\nu} & = R_{at}\nu,
\end{align*}
and $\hat{k} = k$. Substituting these and the result that $(x,y)$ is an integral curve for $V$ we find 
\begin{align*}
\hat{\gamma}_t &= R_{at}(\sigma_1k+\sigma_2)\nu +aR_{\frac{\pi}{2}+at}\gamma  +aR_{\frac{\pi}{2}}(x,y)+(b,c)\\
	&=(\sigma_1\hat{k}+\sigma_2)\hat{\nu} +aR_{\frac{\pi}{2}}\left( R_{at}\gamma +(x,y)\right) +(b,c)\\
	&=  (\sigma_1\hat{k}+\sigma_2)\hat{\nu}+V(\hat{\gamma})
\,.
\end{align*}
\end{proof}

\bibliographystyle{plain}
\bibliography{svg}{}

\end{document}